\documentclass{amsart}
\usepackage{standalone}
\usepackage{graphicx}
\usepackage{amssymb, amsfonts, amsmath, amsthm,mathabx}
\usepackage[all]{xy}
\usepackage{amsmath}
\usepackage{verbatim}
\usepackage[switch, modulo]{lineno}
\usepackage{hyperref}
\usepackage{tikz-cd}
\usepackage{graphicx}
\graphicspath{{../images/}}

\newcommand{\abs}[1]{\left| #1 \right|}

\newcommand{\varep}{\varepsilon}
\newcommand{\B}{\mathcal{B}}

\newcommand{\Sp}{\mathbb{S}}

\newcommand{\R}{\mathbb{R}}
\newcommand{\RP}{\mathbb{RP}}

\newcommand{\tr}{{\rm tr}}

\newcommand{\C}{\mathbb{C}}

\newcommand{\GL}{\operatorname{GL}}
\newcommand{\SL}{\operatorname{SL}}
\newcommand{\SLpm}{\operatorname{SL}_{\pm}}
\newcommand{\PSL}{\operatorname{PSL}}
\newcommand{\SO}{\operatorname{SO}}
\newcommand{\so}{\mathfrak{so}}

\newcommand{\g}{\mathfrak{g}}
\newcommand{\Z}{\mathbb{Z}}

\newcommand{\HH}{\mathbb{H}}

\newcommand{\Sc}{\mathcal{S}}
\newcommand{\Sca}{{\Sc_a}}

\newcommand{\dev}{\mathsf{dev}}
\newcommand{\hol}{\mathsf{hol}}
\renewcommand{\hom}{\operatorname{Hom}}
\newcommand{\rep}{\operatorname{Rep}}

\newcommand{\bs}{\backslash}

\newcommand{\CS}{\mathsf{CS}}
\newcommand{\ad}{\mathsf{ad}}
\newcommand{\Ad}{\mathsf{Ad}}
\renewcommand{\arg}{\mathsf{Arg}}
\newcommand{\rhohyp}{\rho_{hyp}}
\newcommand{\res}{\mathsf{res}}
\renewcommand{\sl}{\mathfrak{sl}}
\renewcommand{\v}{\mathfrak{v}}
\newcommand{\Scusp}{\mathcal{C}}

\newcommand{\m}{\mathfrak{m}}

\newtheorem{theorem}{Theorem}

\newtheorem{proposition}[theorem]{Proposition}
\newtheorem{corollary}[theorem]{Corollary}
\newtheorem{lemma}[theorem]{Lemma}
\newtheorem{remark}[theorem]{Remark}
\newtheorem{definition}{Definition}

\numberwithin{equation}{section}
\numberwithin{theorem}{section}

\begin{document}

\title[Constructing convex projective 3-manifolds with generalized cusps]{Constructing convex projective 3-manifolds with generalized cusps}
\author{Samuel A. Ballas}
\email{ballas@math.fsu.edu}
\date{\today}
\address{Department of Mathematics\\ 
Florida State University\\ Tallahassee, FL 32306, USA}

\maketitle

\begin{abstract}
 We prove that non-compact finite volume hyperbolic 3-manifolds that satisfy a mild cohomological condition (infinitesimal rigidity) admit a family of properly convex deformations of their complete hyperbolic structure where the ends become generalized cusps of type 1 or type 2. We also discuss methods for controlling which types of cusps occur. Using these methods we produce the first known example of a 1-cusped hyperbolic 3-manifold that admits a convex projective structure with a type 2 cusp. We also use these techniques to produce new 1-cusped manifolds that admit a convex projective structure with a type 1 cusp.

\end{abstract}

Unless stated otherwise, all manifolds in this paper are orientable. A subset $\Omega$ of the projective sphere, $\Sp^n$, is \emph{properly convex} if it is a bounded convex subset of some affine subspace of $\Sp^n$. A \emph{properly convex manifold} is a quotient $\Omega/\Gamma$, where $\Omega$ is properly convex and $\Gamma$ is a discrete, torsion-free subgroup of $\SL(n+1,\R)$ that preserves $\Omega$. An important example of a properly convex set is the Klein model of $n$-dimensional hyperbolic space. As a result, complete hyperbolic manifolds provide a broad and important class of properly convex manifolds. 

Suppose $M$ is an $n$-manifold, a \emph{(marked) convex projective structure on $M$} is a pair $(f,N)$, where $N$ is a properly convex manifold and $f:M\to N$ is a diffeomorphism. There is a natural equivalence relation on convex projective structures and the \emph{deformation space} of convex projective structures on $M$, denoted $\mathfrak{B}(M)$, is the set of equivalence classes of convex projective structures. When $M$ is a finite volume hyperbolic manifold and $n\geq 3$ Mostow rigidity implies that there is a distinguished base point in $\mathfrak{B}(M)$ coming from the equivalence class of the complete hyperbolic structure on $M$. A primary focus of this work is to understand the possible geometry of points in a neighborhood of this basepoint.   

We now restrict our discussion to dimension 3. Unlike the hyperbolic setting which is extremely rigid, it is sometimes possible to produce a variety of interesting deformations in the properly convex setting. However, there are some (loose) similarities to the hyperbolic setting. In practice, convex projective structures on \emph{closed} manifold tend to be quite rigid. In \cite{CoLoThist1} Cooper--Long--Thistlethwaite analyzed several thousand 3-manifolds with two-generator fundamental group and found that a vast majority ($>90\%$) do not admit any properly convex deformations of their hyperbolic structure (i.e.\ the hyperbolic structure is an isolated point of $\mathfrak{B}(M)$). However, they also found a small number of examples that admit positive dimensional families of deformations of their complete hyperbolic structure (see \cite{CoLoThist2}). There are also other isolated examples of closed 3-manifolds whose complete hyperbolic structure can be deformed (see \cite{HePo,Benoist-survey,Bal2bridge,ChoiHodLee,ChoiRadial}, for example).

There are also similarities between the deformation theory of hyperbolic and convex projective structures when $M$ is non-compact, but has finite volume. In both settings it is possible to find deformations that are ``supported near the boundary.'' In the hyperbolic setting, it is well known (see \cite{Thu}) that a $k$-cusped hyperbolic manifold admits a (real) $2k$-dimensional family of deformations of its complete hyperbolic structure. However, these deformations only give rise to \emph{incomplete} hyperbolic structures. Loosely speaking, this is a consequence of there not being any way to deform the cusps of $M$ in the category of hyperbolic geometry without losing completeness. 

However in the context of properly convex geometry, \emph{generalized cusps} provided many interesting ways to deform the cusps of a hyperbolic 3-manifold while preserving completeness (with respect to an appropriate metric). Generalized cusps (see Section \ref{gencusp} for precise definitions) are best thought of as properly convex generalizations of cusps of finite volume hyperbolic manifolds. They were first introduced by Cooper--Long--Tillmann \cite{CoLoTiII} and were recently classified by the author, D.\ Cooper, and A.\ Leitner in \cite{BCL}. In dimension 3 generalized cusps come in 4 different flavors (type 0, type 1, type 2, and type 3), where the types interpolate between the holonomy of their fundamental group being unipotent (type 0) and diagonalizable (type 3). The main result of this paper is that when $M$ is \emph{infinitesimally rigid rel.\ $\partial M$} (see Section \ref{infdefs} for definition) it is always possible to find a convex projective structure on $M$ whose ends are all of type 1 or type 2. 

\begin{theorem}\label{mainthm}
	Let $M$ be a finite volume, non-compact hyperbolic 3-manifold. Suppose that $M$ is infinitesimally rigid rel. $\partial M$ then there is a convex projective structures on $M$ where each end is a generalized cusp of type 1 or type 2.
\end{theorem}

Theorem \ref{mainthm} is a consequence of the more general result, which says that when $M$ is infinitesimally rigid rel.\ $\partial M$ it is always possible to deform the hyperbolic structure in $\mathfrak{B}(M)$, while maintaining some control over the geometry near the boundary. 

\begin{theorem}\label{mainthm2}
	Let $M$ be a finite volume, non-compact hyperbolic 3-manifold with $k\geq 1$ cusps and let $\mathfrak{B}(M)$ be deformation space of convex projective structures on $M$. Suppose that $M$ is infinitesimally rigid rel. $\partial M$ then there is a $k$-dimensional family $U\subset \mathfrak{B}(M)$ containing the complete hyperbolic structure on $M$ and consisting of convex projective structures on $M$ whose ends are generalized cusps of type 0, type 1 or type 2. 
	\end{theorem}

While there are infinitely many hyperbolic 3-manifolds that are not infinitesimally rigid, (for instance if $M$ contains a closed totally geodesic surface), in practice, the hypothesis that $M$ is infinitesimally rigidity rel.\ $\partial M$ is not particularly restrictive. For instance, in \cite{HePo}, Heusener--Porti prove that infinitely many 1-cusped manifolds arising as surgery on the Whitehead link are infinitesimally rigid rel.\ $\partial M$. These examples include infinitely many twist knots and infinitely many once-punctured torus bundles with tunnel number 1. Furthermore, numerical computations performed by the author, J.\ Danciger, and G.-S.\ Lee suggest that a majority of manifolds in the SnapPy cusped census \cite{SnapPy} are infinitesimally rigid rel. $\partial M$. 

The proof of Theorem \ref{mainthm2} uses a transversality argument in the space $\hom(\pi_1M,\SL(4,\R))$. The idea is to construct a submanifold $\Sc$ of representations in $\hom(\pi_1 \partial M,\SL(4,\R))$ whose elements are the holonomy representations of generalized cusps of type 0, type 1, and type 2 (see Section \ref{slicesection} for details). We then show that $\Sc$ has transverse intersection with the image of a certain ``restriction map'' in order to construct representations in $\hom(\pi_1M,\SL(4,\R))$. We then use a version of the Ehresmann--Thurston principle for properly convex structures due to Cooper--Long--Tillmann \cite{CoLoTiII} in order to show that these representations are holonomies of convex projective structures on $M$ with ends that are generalized cusps.  

One application of this theorem is to complete the picture of which generalized cusp types can occur as ends of a convex projective structure on a 1-cusped hyperbolic manifold. Type 0 cusps occur as the ends of finite volume hyperbolic 3-manifolds, and so there are many examples coming from the classical theory of hyperbolic geometry. At the other end of the spectrum the author, along with J.\ Danciger and G.-S.\ Lee (see \cite{BDL}) prove a complementary result which shows that under the same hypothesis as Theorem \ref{mainthm2}, it is possible to find infinite families of convex projective structures on $M$ with type 3 cusps. In particular, it is possible to produce 1 cusped 3-manifolds that admit convex projective structures with type 3 cusps.

However, up to this point there have only been isolated examples of manifolds with type 1 or type 2 cusps. One such example is given  by the author in \cite{BalFig8}, where it is shown that the complement in $S^3$ of the figure-eight knot admits a convex projective structure with a type 1 cusp. Until very recently, there were no known examples of a hyperbolic 3-manifold with type 2 cusps. However, the author was recently made aware of work of M.\ Bobb \cite{Bobb} in which he produces the first examples of hyperbolic 3 manifolds with a cusp of type 2. His methods are quite different than those of this paper and involve simultaneously bending along multiple embedded totally geodesic hypersurfaces. However, he uses arithmetic methods to produce examples with many totally geodesic hypersurfaces, and as a result, the examples he constructs are arithmetic and have many cusps. 

In Section \ref{controlcusps} we analyze the geometry of the ends produced by Theorem \ref{mainthm2}. Using these result we are able to show that the complement in $S^3$ of the $5_2$ knot admits a convex projective structure with a type 2 cusp (see Theorem \ref{52type2defs}). To the best of the author's knowledge, this is the first known 1-cusped manifold that admits a convex projective structure with a type 2 cusp. Moreover, in Theorem \ref{slicecoordtyperelationship} we show that a ``generic'' deformation constructed by Theorem \ref{mainthm2} will have only type 2 cusps, so in practice Theorem \ref{mainthm2} should produce infinitely many new examples of 1-cusped manifolds that admit a convex projective structure with a type 2 cusp. 

Despite the genericity of type 2 cusps, it is still possible to use Theorem \ref{mainthm2} to produce examples of properly convex manifolds with type 1 cusps. Specifically, we show in Section \ref{controlcusps} that if $M$ satisfies the hypotheses of Theorem \ref{mainthm2} and admits a certain type of orientation reversing symmetry then Theorem \ref{mainthm2} produces convex projective structures on $M$ whose cusps are all of type 1. We then apply this result to show that the complement in $S^3$ of the $6_3$ knot admits a convex projective structure with a type 1 cusp (see Theorem \ref{63type1}).

\subsection*{Organization of the paper} Section \ref{propconvgeom} provides some background and definitions related properly convex geometry, generalized cusps, and deformations of convex projective structures. Section \ref{infdefs} discusses infinitesimal deformations and their relationship to twisted cohomology. It also provides some relevant cohomological results in dimension 3. Section \ref{slicesection} defines the slice that will be used in the main transversality argument and proves several of its important properties. Section \ref{transargument} is the technical heart of the paper. In this section we provide the main transversality argument and prove Theorem \ref{mainthm2}. Section \ref{controlcusps} provides the necessary tools to analyze the geometry of the cusps for the deformations produced by Theorem \ref{mainthm2}. In particular it provides the ingredients to prove Theorem \ref{mainthm}. Finally, Section \ref{examples} outlines the computations necessary to prove the results concerning the $5_2$ knot and $6_3$ knot.

 \subsection*{Acknowledgements} The author would like to thank the anonymous referee for several helpful suggestions, including the addition of the section on obstruction theory. The author would also like to thank Joan Porti for several useful discussions about computing obstructions. The author was partially supported by NSF grant DMS 1709097. The author also acknowledges support from the National Science Foundation grants DMS 1107452, 1107263, 1107367 “RNMS: GEometric structures And Representation varieties” (the GEAR Network).

\section{Properly convex geometry}\label{propconvgeom}

The \emph{projective $n$-sphere}, denoted $\Sp^n$, is the space of rays through the origin in $\R^{n+1}$. More concretely, $\Sp^n=(\R^{n+1}\bs \{0\})/\sim$ where $x\sim y$ if an only if there is $\lambda>0$ such that $x=\lambda y$. The group $\GL(n+1,\R)$ acts on $\Sp^n$, however, this action is not faithful. The kernel of the action is $\R^+\operatorname{I}$. For each class in $\GL(n+1,\R)/\R^+\operatorname{I}$ there is a unique representative with determinant $\pm 1$. Therefore, if we let  
$$G=\SL_\pm(n+1,\R):=\{A\in \GL(n+1,\R)\mid \det(A)=\pm1\}$$
then there is a natural identification of $\GL(n+1,\R)/\R^+\operatorname{I}\cong G$, and $G$ is the full group of projective automorphisms of $\Sp^n$.

 The projective $n$-sphere is related to the more familiar \emph{real projective $n$-space}, denoted $\RP^n$, which consist of lines through the origin in $\R^{n+1}$ via the 2--to--1 covering given by mapping a ray to the line that contains it. It is possible to work entirely with $\RP^n$ instead of $\Sp^n$, however the benefit of working with $\Sp^n$ is that it is orientable for all $n$ and its group of projective automorphisms consists of matrices instead of equivalences classes of matrices. This allows one to use tools from linear algebra, such as eigenvalues, traces, etc., without having to worry about picking representative from equivalence classes.

 A \emph{projective hyperplane}, or \emph{hyperplane} for short, is the image of an $n$-dimensional subspace of $\R^{n+1}$ in $\Sp^n$. In other words, a projective hyperplane is a great $(n-1)$-sphere. If $H$ is a projective hyperplane then either hemisphere of  $\Sp^n\bs H$ is naturally identified with $\mathbb{A}^n$ and is thus called an \emph{affine patch} (see Figure \ref{patch}). The group $G$ acts transitively on the set of affine patches, and so there is model for an affine patch given by 
 $$\mathbb{A}^n=\{[x_1:\ldots,x_n:1]\mid x_i\in \R\},$$
  where $[x_1:\ldots:x_{n+1}]$ is the \emph{homogeneous coordinate} for the ray containing the point $(x_1,\ldots, x_{n+1})\in \R^{n+1}$. The stabilizer in $\GL(n+1,\R)$ of this affine patch is \emph{affine group}, denoted $G_A$, and consists of matrices that can be written in block form as
  $$\begin{pmatrix}
  	A & b\\
  	0 & 1
  \end{pmatrix},$$
  where $A\in \GL(n,\R)$, $b\in \R^n$.  The group $G_A$ acts faithfully on $\mathbb{A}^n$. 

\begin{center}
	\begin{figure}
		\includegraphics[scale=.25]{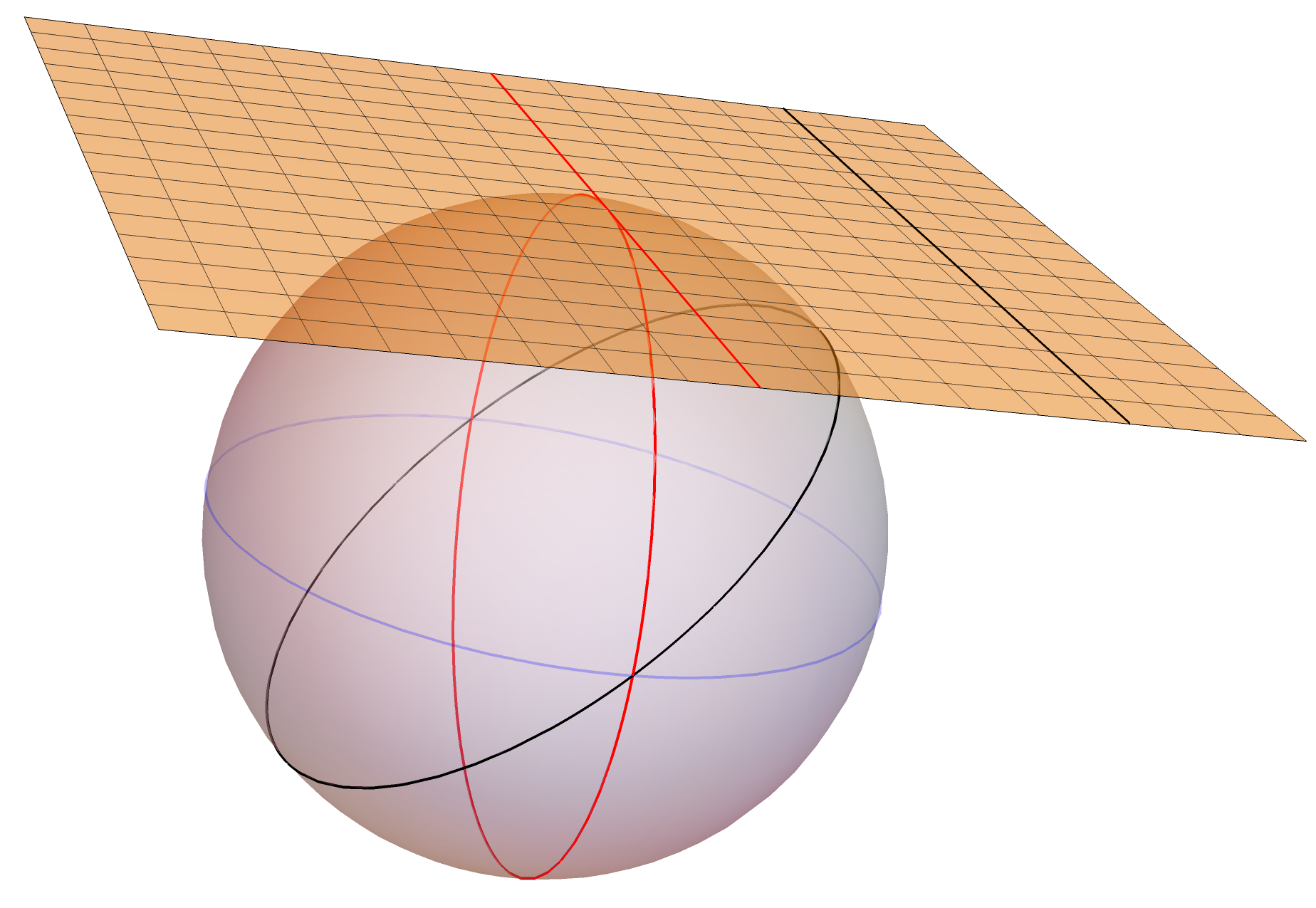}
		\caption{\label{patch} An affine in $\Sp^n$ is identified with $\R^n$ via radial projection}
	\end{figure}
\end{center}

Let $\Omega\subset \Sp^n$ with non-empty interior, then $\Omega$ is \emph{properly convex} if the topological closure, $\overline{\Omega}$, of $\Omega$ is a convex subset of some affine patch. Every properly convex set $\Omega$ comes with a group $\SL(\Omega)$ consisting of elements of $G$ that preserve $\Omega$. 

If $\Omega$ is properly convex and $\Gamma\leq \SL(\Omega)$ is discrete and torsion-free then $\Omega/\Gamma$ is a \emph{properly convex manifold}. An important example to keep in mind is the following: let $\mathcal{C}$ be a component of the interior in $\R^{n+1}$ of the light cone of a quadratic form of signature $(n,1)$ and let $\Omega=\mathcal{C}\cap \Sp^n$. This is the well known \emph{Klein model} of hyperbolic $n$-space. In this setting, $\SL(\Omega)=O^+(n,1)\cong \operatorname{Isom}(\HH^n)$ and we see that complete hyperbolic manifolds are examples of properly convex manifolds. 

\subsection{Generalized cusps in projective manifolds}

Let $M$ be a finite volume hyperbolic $n$-manifold. The thick-thin decomposition allows one to decompose $M$ into $M=M_K\sqcup \partial$, where $M_K$ is a compact manifold (possibly with boundary) homotopy equivalent to $M$ and $\partial=\sqcup_{i=1}^k\partial_i$ is a union of finitely many \emph{cusps}, where each $\partial_i$ is diffeomorphic to $E_i\times (0,\infty)$ for some closed Euclidean $(n-1)$-manifold $E_i$. As a result, $\Delta_i:=\pi_1(\partial_i)=\pi_1(E_i)$ is virtually abelian. It is also possible to describe the geometry of hyperbolic cusps: For each $t\in (0,\infty)$, $E_i\times\{t\}$ is a strictly convex hypersurface in $\partial_i$. Specifically, the universal cover of $E_i\times \{t\}$ can be identified with a horosphere in $\HH^n$. Motivated by the previous discussion of cusps in hyperbolic manifolds we make the following definition:
\begin{definition}\label{gencusp}
	A properly convex $n$-manifold, $C=\Omega/\Gamma$ is a \emph{generalized cusp} if 
	\begin{itemize}
		\item $\Gamma$ is virtually abelian
		\item $C\cong E\times (0,\infty)$, where $E$ is a closed Euclidean $(n-1)$-manifold
		\item For each $t\in (0,\infty)$, the universal cover of $E\times \{t\}$ in $\Omega$ is strictly convex.	\end{itemize}
\end{definition}

The previous discussion shows that cusps of finite volume hyperbolic $n$-manifolds are generalized cusps. Generalized cusps were originally introduced in \cite{CoLoTiII} (using a slightly different definition) where they are instrumental in understanding properly convex deformations of non-compact manifolds. The current definition of generalized cusps is the one given by Cooper, Leitner, and the author in \cite{BCL}. In this work it is shown that the two definitions of generalized cusps are, in fact, equivalent. 

The main result from \cite{BCL} is a classification result for generalized cusps in each dimension. Before providing some specific examples we roughly explain the classification result.  In dimension $n$ there are $n+1$ \emph{types} of cusps which are denoted type 0 through type $n$. Each type determines an $n$-dimensional Lie subgroup of $\GL(n+1,\R)$, $T_k$ (where $k$ is the type), called the \emph{enlarged translation group} which is isomorphic to $\R^n$. Roughly speaking, the larger the type, the closer the enlarged translation group is to being diagonalizable. If $C=\Omega/\Gamma$ is a generalized cusp of type $k$ then $\Gamma$ contains a finite index subgroup $\Gamma'$ that is a lattice in a certain codimension 1 Lie subgroup (depending on $\Gamma$) of $T_k$. 

We now explain the classification in detail in the case where $n=3$. Since the torus is the only closed Euclidean surface it follows that each $3$-dimensional generalized cusp is diffeomorphic to $T^2\times (0,\infty)$. In this case there are 4 types of generalized cusp, and we will primarily concern ourselves with type 0, type 1, and type 2 cusps. For many purposes, it is simpler to work with the Lie algebra $\mathfrak{t}_k$ of the enlarged translation group $T_k$. Nothing is lost working with $\mathfrak{t}_k$ since $\mathfrak{t}_k$ and $T_k$ are isomorphic via the exponential map.  
\subsubsection{Type 0 cusps}
Let $x,y,z\in \R$, then the Lie algebra $\mathfrak{t}_0$ consists of elements of the form
\begin{equation}\label{type0liealg}
	m_0(x,y,z)=\begin{pmatrix}
	0 & x & y & z\\
	0 & 0 & 0 & x\\
	0 & 0 & 0 & y\\
	0 & 0 & 0 & 0
\end{pmatrix},
\end{equation}

and $T_0$ consists of elements of the form
$$M_0(x,y,z)=\exp(m_0(x,y,z))\begin{pmatrix}
1 & x & y & z+\frac{x^2+y^2}{2}\\
0 & 1 & 0 & x\\
0 & 0 & 1 & y\\
0 & 0 & 0 & 1	
\end{pmatrix}
.$$

Consider the codimension 1 subgroup $T(0)$ of $T_0$ consisting of elements of the form $M_0(x,y,0)$. When regarded as elements of $G_A$, $T(0)$ preserves the properly convex set 
$$\Omega_0=\left\{[a:b:c:1]\in \Sp^3\mid a>\frac{b^2+c^2}{2}\right\}.$$
 For $s>0$ let 
 $$\mathcal{H}_s^0=\left\{[a:b:c:1]\in \Sp^3\mid a=\frac{b^2+c^2}{2}+s\right\}.$$
  Each $\mathcal{H}_s$ is also $T(0)$-invariant and the $\mathcal{H}_s^0$ give a codimension 1 foliation of $\Omega_0$ by strictly convex hypersurfaces. A \emph{type 0 generalized cusp} is a properly convex manifold that is projectively equivalent to $\Omega_0/\Gamma$ where $\Gamma$ is a lattice in $T(0)$. Such manifolds are easily seen to be generalized cusps since $\mathcal{H}_s^0/\Gamma$ provides a foliation of $\Omega/\Gamma$ by strictly convex tori.

  This is a familiar construction in the context of hyperbolic geometry: $\Omega_0$ is the paraboloid model of $\HH^3$ (see \cite[\S 3]{CoLoTi}) and the foliation $\mathcal{H}_s^0$ is a foliation of $\HH^3$ by concentric horospheres. The group $T(0)$ consists of parabolic isometries of $\HH^3$ with a common fixed point on $\partial \HH^3$, and $\Omega_0/\Gamma$ is a hyperbolic torus cusp.  
  
  \begin{center}
	\begin{figure}
		\includegraphics[scale=.4]{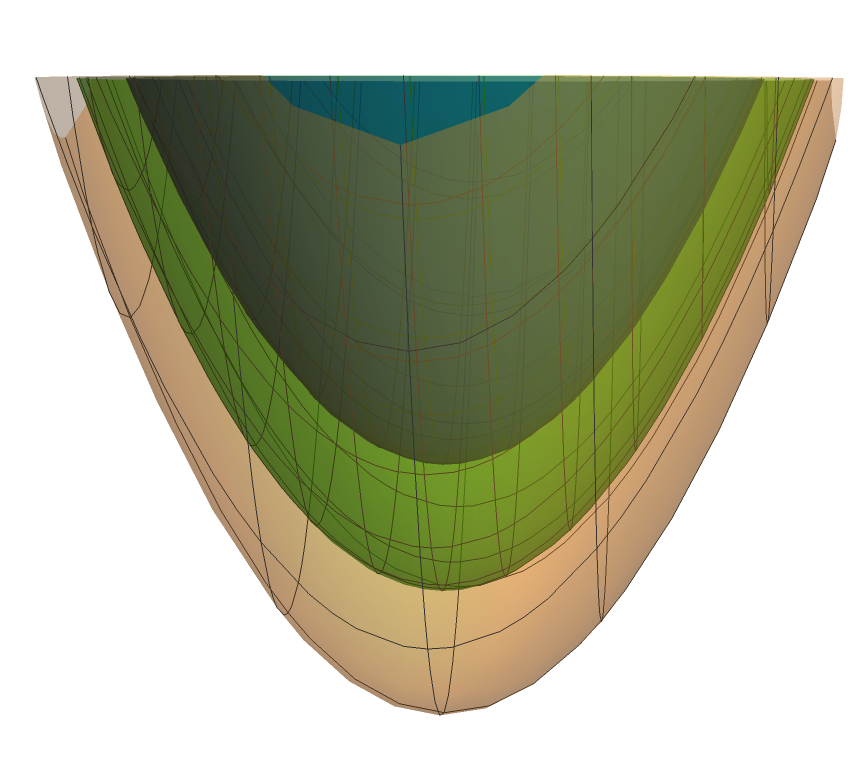}
		\caption{\label{type0} $\HH^3$ along with some leaves of the horosphere foliation viewed in an affine patch}
	\end{figure}
\end{center}

\subsubsection{Type 1 cusps}\label{type1cuspsection}

Again, let $x,y,z\in \R$, then the Lie algebra $\mathfrak{t}_1$ consists of elements of the form 
\begin{equation}\label{type1liealg}
	m_1(x,y,z)=\begin{pmatrix}
	x & 0 & 0 & 0\\
	0 & 0 & y & z\\
	0 & 0 & 0 & y\\
	0 & 0 & 0 & 0 
\end{pmatrix},
\end{equation}

 and let $T_1$ consist of element of the form 
 $$M_1(x,y,z)=\exp(m_1(x,y,z))=\begin{pmatrix}
 	e^x & 0 & 0 & 0\\
 	0 & 1 & y & z+\frac{y^2}{2}\\
 	0 & 0 & 1 & y\\
 	0 & 0 & 0 & 1
 \end{pmatrix}.$$
 
 Let $\lambda\neq 0$ and let $T(\lambda)$ be the codimension 1 subgroup of $T_1$ consisting of elements of the form $M_1(\lambda x,y,-\lambda^{-1}x)$. For any $\lambda\neq 0$, the group $T(\lambda)$ preserves both the properly convex set
 $$\Omega_1:=\left\{[a:b:c:1]\in \Sp^3\mid a>0,\ b>\frac{c^2}{2}-\lambda^{-2}\log(a)\right\}$$
 and the strictly convex codimension 1 foliation of $\Omega_1$ by 
 $$\mathcal{H}^1_s:=\left\{[a:b:c:1]\in \Sp^3\mid a>0,\ b=\frac{c^2}{2}-\lambda^{-2}\log(a)+s\right\},\ s>0$$
\begin{center}
	\begin{figure}
		\includegraphics[scale=.5]{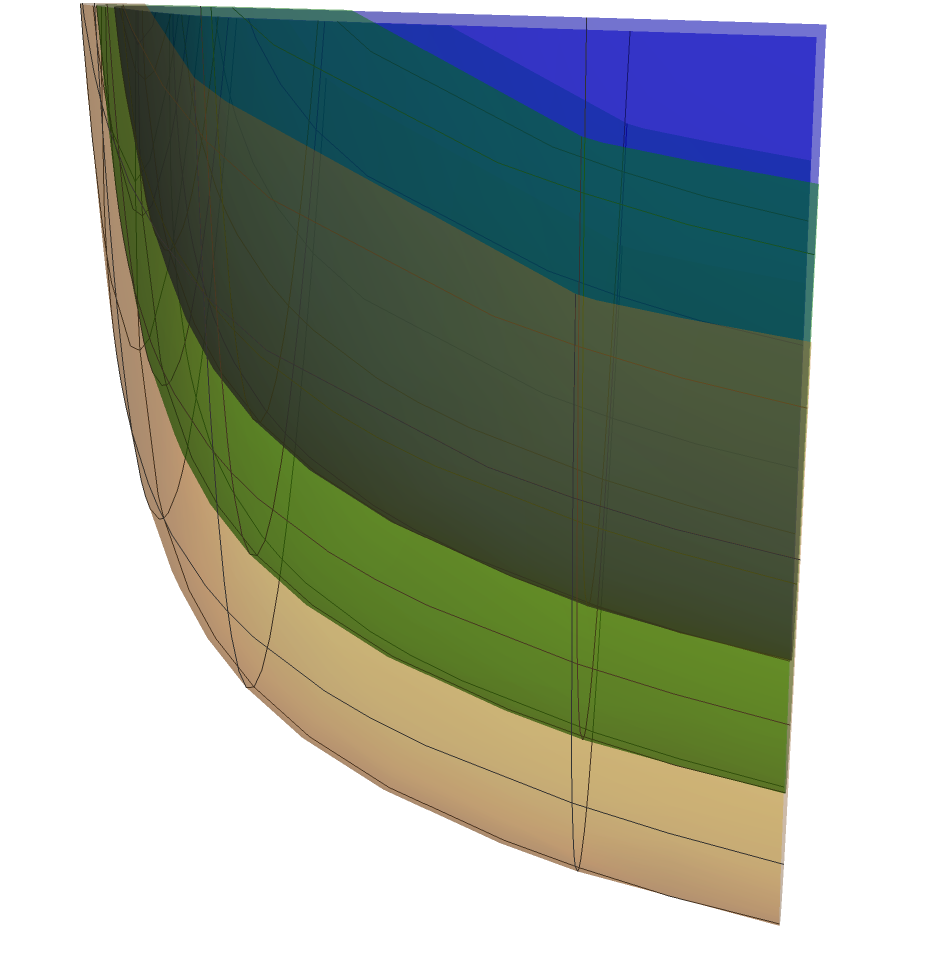}
		\caption{\label{type1} The domain $\Omega_1$ and a few leaves of the foliation $\mathcal{H}^1_s$ in an affine patch $d=1$.}
	\end{figure}
\end{center}

A \emph{type 1 generalized cusp} is a properly convex manifold that is projective equivalent to $\Omega_1/\Gamma$ where $\Gamma$ is a lattice in $T(\lambda)$ for some $\lambda\neq 0$. Again, such manifolds are easily seen to be generalized cusps since $\mathcal{H}_s^1/\Gamma$ provides a foliation of $\Omega_0/\Gamma$ by strictly convex tori.

\subsubsection{Type 2 cusps}\label{type2cuspsection}

Once again, let $x,y,z\in \R$, then the Lie algebra $\mathfrak{t}_2$ consists of elements of the form
\begin{equation}\label{type2liealg}
	m_2(x,y,z)=\begin{pmatrix}
		x & 0 & 0 & 0\\
		0 & y & 0 & 0\\
		0 & 0 & 0 & z\\
		0 & 0 & 0 & 0
	\end{pmatrix},
\end{equation}

and let $T_2$ consist of elements of the form
$$M_2(x,y,z)=\exp(m_2(x,y,z))=\begin{pmatrix}
	e^x & 0 & 0 & 0\\
	0 & e^y & 0 & 0\\
	0 & 0 & 1 & z\\
	0 & 0 & 0 & 1
\end{pmatrix}.$$

Let $\lambda_1,\lambda_2\in \R$ such that $\lambda_1\lambda_2>0$ and let $T(\lambda_1,\lambda_2)$ be the codimension 1 subgroup of $T_2$ consisting of elements of the form $M_2(\lambda_1x,\lambda_2y,-\lambda_1^{-1}x-\lambda_2^{-1}y)$. Each $T(\lambda_1,\lambda_2)$ preserves both the properly convex set 
$$\Omega_2=\left\{[a:b:c:1]\in \Sp^3\mid a,b>0,\ c>-\lambda_1^{-2}\log(a)-\lambda_2^{-2}\log(b)\right\}$$
and the strictly convex codimension 1 foliation
$$\mathcal{H}^2_s=\left\{[a:b:c:1]\in \Sp^3\mid a,b>0,\ c=-\lambda_1^{-2}\log(a)-\lambda_2^{-2}\log(b)+s\right\},\ s>0$$

A \emph{type 2 generalized cusp} is a properly convex manifold that is projectively equivalent to $\Omega_2/\Gamma$ where $\Gamma$ is a lattice in $T(\lambda_1,\lambda_2)$ for some $\lambda_1,\lambda_2\in \R$ with $\lambda_1\lambda_2>0$. As before, these manifolds are easily seen to be generalized cusps. 
\begin{center}
	\begin{figure}
		\includegraphics[scale=.4]{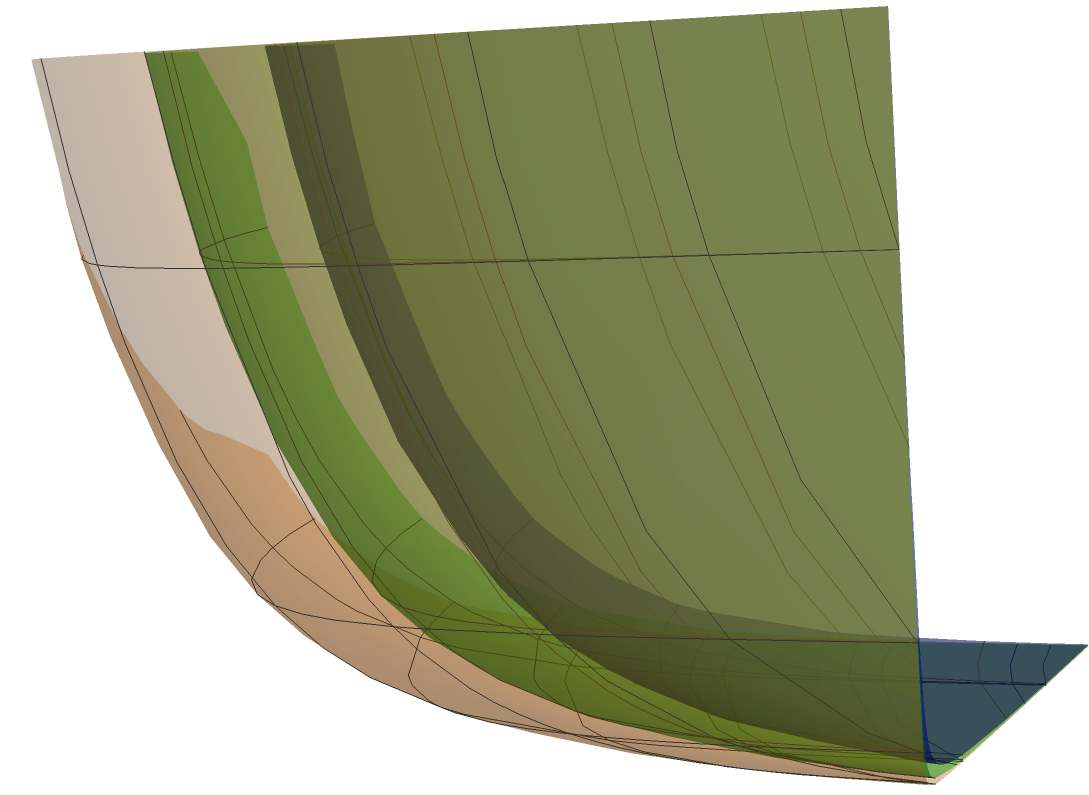}
		\caption{\label{type2} The domain $\Omega_2$ and a few leaves of the foliation $\mathcal{H}^2_s$ in an affine patch $d=1$.}
	\end{figure}
\end{center}

\begin{remark}
If $\lambda_1\lambda_2<0$ then it is still possible to define $\mathcal{H}^2_s$, however, in this case the horospheres are not strictly convex and $\Omega_2$ is not properly convex.  	
\end{remark}

\subsection{Deformation space of convex projective structures}

Let $N$ be the interior of a compact manifold (for instance a finite volume hyperbolic $n$-manifold) and let $\Gamma=\pi_1 N$. A \emph{(marked) convex projective structure} on $N$ is a pair $(f,M)$ where $M=\Omega/\Gamma$ is a properly convex manifold and $f:N\to M$ is a diffeomorphism called a \emph{marking}. Lifting the marking to the universal cover gives a diffeomorphism $\dev:\tilde N\to \Omega$, called a \emph{developing map}. The marking also induces a representation $\rho:\Gamma \to \SL(\Omega)\subset G$ given by $\rho=f_\ast$ called a \emph{holonomy representation}. 

We now define an equivalence relation on marked convex projective structures. Given two marked convex projective structures $(f,M)$ and $(f',M')$ on $N$ with developing maps $\dev$ and $\dev'$, we say that $(f,M)\sim (f',M')$ if there is a submanifold $N_0\subset N$  obtained by removing a collar of $\partial N$ and an element $g\in G$ such that the following diagram computes, up to isotopy. 

$$
\begin{tikzcd}
 & \dev(N_0)\arrow["g",dd]\\
 N_0\arrow["\dev", ur]\arrow["\dev'"',dr] & \\
  & \dev'(N_0)
\end{tikzcd}
$$

  In other words, there is a projective bijection from the complement of a collar of the boundary of $M$ to the complement of a collar of the boundary of $M'$. If $\rho$ and $\rho'$ are the holonomy representations of $(f,M)$ and $(f',M')$ then $\rho'=g\rho g^{-1}$, and so we see that equivalent marked convex projective structures have conjugate holonomy representations. The \emph{deformation space of convex projective structures on $N$}, denoted $\mathfrak{B}(N)$, is the set of marked convex projective structures on $N$, modulo the above equivalence 
  
  Let $\rep (\Gamma,G):=\hom(\Gamma,G)/G,$ where the action of $G$ is by conjugation. For most purposes, it suffices to regard $\rep(\Gamma,G)$ as given by the naive topological quotient. However, it will sometimes be necessary to endow $\rep(\Gamma,G)$ with the structure of an affine variety (at least locally). In order to endow $\rep(\Gamma,G)$ with this type of structure it is necessary to use the categorical quotient (see \cite{LubMag} for details). In general these quotients are not the same, however near the representations we will need to consider these two quotients are locally homeomorphic topological spaces.
  
   By the above discussion, there is a map $\hol :\mathfrak{B}(N)\to \rep(\Gamma,G)$, called the \emph{holonomy map}, that associates to an equivalence class of convex projective structures the conjugacy class of its holonomy representation. Using the weak $C^\infty$ topology on developing maps allows us to endow $\mathfrak{B}(N)$ with a topology for which $\hol$ is continuous. 
  
  
  We now restrict our attention to the case where $N$ is a finite volume hyperbolic $n$-manifold, and we let $\rhohyp:\Gamma\to G$ be the holonomy representation of a marked hyperbolic structure on $N$. If $N$ is closed then work of Koszul \cite{Ko} shows that $\hol$ is an open map. In other words, if $\rho:\Gamma\to G$ is a representation and $[\rho]$ is sufficiently close to $[\rhohyp]$ in $\rep(\pi_1N,G)$ then $\rho$ is also the holonomy of a marked convex projective structure on $N$. This idea is useful since it reduces the geometric problem of deforming marked convex projective structures on $N$ to the simpler algebraic problem of deforming $[\rhohyp]$ in $\rep(\Gamma,G)$. 
  
  When $N$ is non-compact, Koszul's result breaks down. For example, if $N$ is a cusped hyperbolic $3$-manifold then there are representations arbitrarily close to $[\rhohyp]$ in $\rep(\Gamma,G)$ that correspond to incomplete hyperbolic structures on $N$. It is easily seen that these are not holonomies of marked convex projective structures on $N$ (for instance they are either indiscrete or non-faithful). However, in recent work of Cooper--Long--Tillmann \cite{CoLoTiII} are able to prove a relative version of Koszul's theorem. More precisely, they show that small deformations at the level of representations that preserve certain boundary conditions are guaranteed to be the holonomy of a convex projective structure on $N$. In order to state their precise result we need to introduce some terminology. 
  
  Let $\hom_{ce}(\Gamma,G)$ be the representations of $\Gamma$ into $G$ that are holonomies of convex projective projective structures on $N$ such that each end of $N$ is a generalized cusp. A group $\Delta\subset G$ is a \emph{virtual flag group} if it contains a finite index subgroup that is conjugate in $G$ to an upper-triangular group. For instance, the image of the holonomy of a generalized cusp is a virtual flag group. 
The following is a paraphrasing of part of Theorem 0.2 from \cite{CoLoTiII}.
\begin{theorem}\label{cltdefthm}
	Suppose $W$ is a compact, connected $n$-manifold, let $N=W\backslash \partial W$, and let $\{V_1,\ldots,V_k\}$ be the set of connected components of $\partial W$. Let $B_i\cong V_i\times [0,1)$ be the end of $N$ corresponding to $V_i$. Suppose that $\rho_0\in \hom_{ce}(\Gamma,G)$ and for $t\in (-1,1)$, $\rho_t:\Gamma\to G$ is a continuous path of representation with the property that $\rho_t(\pi_1 B_i)$ is a virtual flag group for each $i$. Then there is $\varep>0$ such that for $t\in (-\varep,\varep)$, $\rho_t\in \hom_{ce}(\Gamma,G)$.  
\end{theorem}

Informally, this theorem says that if one performs a small deformation of the holonomy of a properly convex projective structure on $N$ with generalized cusps ends, subject to the constraint that the image of the peripheral subgroups remain virtual flag groups, then the resulting representation is also the holonomy of a properly convex projective structure on $N$ with generalized cusp ends.

\section{Infinitesimal deformations and twisted cohomology}\label{infdefs}
Let $\Gamma$ be a finitely generated group and let $G$ be a Lie group with Lie algebra $\g$. Let $\hom(\Gamma,G)$ be the set of homomorphisms from $\Gamma$ to $G$. This set is called the \emph{representation variety of $\Gamma$ into $G$}, or just representation variety if $\Gamma$ and $G$ are clear from context. If $\Gamma$ is generated by elements $\gamma_1,\ldots, \gamma_k$ then $\hom(\Gamma,G)$ can be regarded as a subset of $G^{k}$. The relations in $\Gamma$ give rise to polynomials in the entries of the elements of $G$ and thus $\hom(\Gamma,G)$ is an algebraic subset of $G^k$. Let $\rho_t$ be a smooth path of representations through $\rho_0$ in $\hom(\Gamma,G)$  
Define $u:\Gamma\to \g$  by the formula
\begin{equation}\label{cocycleFormula}
	\gamma \mapsto \left.\frac{d}{dt}\right |_{t=0}\rho_t(\gamma)\rho_0(\gamma)^{-1}
\end{equation}  Let $Z^1(\Gamma,\g^{\rho_0})$ be the set of 1-cocycles with coefficients in $\g$ twisted by the adjoint of $\rho_0$. More precisely, $Z^1(\Gamma,\g^{\rho_0})$ is the set of functions $v:\Gamma\to \g$ with the property that 
$$v(\gamma_1\gamma_2)=v(\gamma_1)+\rho_0(\gamma_1)\cdot v(\gamma_2),$$
 where the action is given by the composition of $\rho_0$ and the adjoint action. Using this formula, it follows that a cocycle is determined by its values of a generating set. The homomorphism condition on $\rho_t$ implies that $u\in Z^1(\Gamma,\g^{\rho_0})$ and we say that $u$ is \emph{tangent to $\rho_t$ at $\rho_0$}. For this reason, we will refer to elements of $Z^1(\Gamma,\g^{\rho_0})$ as \emph{infinitesimal deformations of $\rho_0$}. As mentioned, the space $\hom(\Gamma,G)$ is an algebraic variety and the above construction gives an identification of $Z^1(\Gamma,\g^{\rho_0})$ and $T_{\rho_0}\hom(\Gamma,G)$, where the latter is the \emph{Zariski tangent space} to $\hom(\Gamma,G)$ at $\rho_0$ (see \cite{LubMag} for details). In general, $\rho_0$ will not be a smooth point of $\hom(\Gamma,G)$ (i.e. $\hom(\Gamma,G)$ need not be a smooth manifold near $\rho_0$). The failure of smoothness near $\rho_0$ manifests itself in the existence of infinitesimal deformations that are not tangent at $\rho_0$ to any smooth path in $\hom(\Gamma,G)$. Thus if we wish to prove that $\hom(G,\Gamma)$ is smooth at $\rho_0$ then we must show that every infinitesimal deformation is tangent to a smooth path through $\rho_0$.  
 
 The space $Z^1(\Gamma,\g^{\rho_0})$ consists of 1-cocyles of the chain complex $C^\ast(\Gamma,\g^{\rho_0})$ coming from the group cohomology of $\Gamma$ with coefficients in the $\Gamma$-module $\g^{\rho_0}$ (see \cite{BrownCohom} for more details). The cohomolgy groups of this chain complex will be referred to as $H^\ast(\Gamma,\g^{\rho_0})$. In the case that $\Gamma$ is the fundamental group of an aspherical manifold (e.g.\ a finite volume hyperbolic manifold) then the group cohomology $H^\ast(\Gamma,\g^{\rho_0})$ is naturally isomorphic to the singular cohomology with twisted coefficients, $H^\ast(M,\g^{\rho_0})$. This identification is quite useful, as it allows one to use techniques such as Poincar\'e duality in order to compute group cohomology.  
 
 \subsection{Cohomology of 3-manifolds}\label{3mfldCohomology}

In this section we discuss the cohomology (with twisted coefficients) of hyperbolic 3-manifolds. Throughout this section let $M$ be a finite volume hyperbolic 3-manifold (typically non-compact), let $\Gamma=\pi_1M$, $G=\SL(4,\R)$, and let $\g=\mathfrak{sl}(4,\R)$ be its Lie algebra. Note that since $M$ is aspherical the cohomology group $H^\ast(\Gamma,\g^{\rhohyp})$ is naturally isomorphic to $H^\ast(M,\g^{\rhohyp})$. Let 
\begin{equation}\label{31form}
	J=\begin{pmatrix}
	0 & 0 & 0 & -1\\
	0 & 1 & 0 & 0\\
	0 & 0 & 1 & 0\\
	-1 & 0 & 0 & 0
\end{pmatrix},
\end{equation}
and let $SO(3,1)=\{A\in \SL(4,\R)\mid A^tJA=J\}$. In this setting, there is a representation $\rhohyp:\Gamma\to \SO(3,1)\subset G$ given by the holonomy of the complete hyperbolic structure on $M$. By Mostow rigidity, this representation is unique up to conjugacy in $G$.  

There is also a useful splitting of $\g$ (as an $SO(3,1)$-module). The group $SO(3,1)$ acts on $\g$ via the adjoint action (i.e.\ if $g\in \SO(3,1)$ and $a\in \g$ then $g\cdot a=gag^{-1}$). The map $a\mapsto -Ja^tJ$ is an $SO(3,1)$-module isomorphism. This map is an involution and whose 1-eigenspace is 
$$\so(3,1)=\{a\in \g\mid a^tJ+Ja=0\}$$
 and the we denote the -1-eigenspace by $\v$. This gives a splitting
\begin{equation}\label{slsplitting}
	\g=\so(3,1)\oplus \v.
\end{equation}
Observe that this is only a splitting of $\SO(3,1)$-modules and not of Lie algebras since $\v$ is not closed under Lie brackets. The above construction can be repeated using other symmetric matrices, $J'$ of signature $(3,1)$. Using $J'$ will result in a new splitting of $\sl(4,\R)$ that differs from the original splitting by a conjugacy in $G$. For instance, when executing some of the computation in Section \ref{examples} it is convenient to use a slightly different form. 

The splitting \eqref{slsplitting} induces a splitting at the level of cohomology:
\begin{equation}\label{cohomsplitting}
H^1(\Gamma,\g^{\rhohyp})\cong H^1(\Gamma,\so(3,1)^{\rhohyp})\oplus H^1(\Gamma,\v^{\rhohyp})	
\end{equation}

as well as maps $\pi_{\so(3,1)}:H^1(\Gamma,\g^{\rhohyp})\to H^1(\Gamma,\so(3,1)^{\rhohyp})$ and $\pi_\v:H^1(\Gamma,\g^{\rhohyp})\to H^1(\Gamma,\v^{\rhohyp})$. 

A useful way to understand the cohomology groups of our 3-manifold is by restricting to the boundary. Suppose that $M$ has $k$ cusps and let  $\partial:=\sqcup_{i=1}^k\partial_i$, where $\partial_i$ is the $i$th cusp of $M$. If $\Delta_i=\pi_1\partial_i$, then for each $i$, there is a restriction map, $\res_i:\hom(\Gamma,G)\to \hom(\Delta_i,G)$ given by regarding $\Delta_i$ as a subgroup of $\Gamma$ and restricting representations. By abuse of notation we will denote $\res_i\rhohyp$ by $\rhohyp$. Each of the above maps descends to $(\res_i)_\ast:H^1(\Gamma,\g^{\rhohyp})\to H^1(\Delta_i,\g^{\rhohyp})$. Next, define $\Delta=\oplus_{i=1}^k\Delta_i$, then we have a canonical identification $
\hom(\Delta,G)\cong \bigtimes_{i=1}^k \hom(\Delta_i,G)$. Define $Z^1(\Delta,\g^{\rhohyp}):=\bigoplus_{i=1}^k Z^1(\Delta_i,\g^{\rhohyp})$ and define $B^1(\Delta,\g^{\rhohyp})$ and $H^1(\Delta,\g^{\rhohyp})$ in a similar fashion. Taking the direct sum of the above maps gives 
$$\oplus_{i=1}^k(\res_i)_\ast=:\res_\ast:Z^1(\Gamma,\g^{\rhohyp})\to Z^1(\Delta,\g^{\rhohyp}).$$

 This map sends $B^1(\Gamma,\g^{\rhohyp})$ into $B^1(\Delta,\g^{\rhohyp})$, and thus descends to a map which, by abuse, we denote $\res_\ast:H^1(\Gamma,\g^{\rhohyp})\to H^1(\Delta,\g^{\rhohyp})$. The map $\res_\ast$ respects the splitting \eqref{slsplitting} and we get corresponding maps which by further abuse of notation we call $\res_\ast:H^1(\Gamma,\so(3,1)^{\rhohyp})\to H^1(\Delta,\so(3,1)^{\rhohyp})$ and $\res_\ast:H^1(\Gamma,\v^{\rhohyp})\to H^1(\Delta,\v^{\rhohyp})$.
 
 Next, for each $1\leq i\leq k$, let $m_i\subset \partial_i$ be a homotopically non-trivial curve and let $\mu\subset \Delta$ be the direct sum of the cyclic subgroups generated by the homotopy classes of the $m_i$. Using a construction similar to the previous paragraph, there is again a restriction map which we abusively call $\res_\ast:H^1(\Gamma,\g^{\rhohyp})\to H^1(\mu,\g^{\rhohyp})$. As above there are also analogous maps with coefficients in either $\so(3,1)$ or $\v$.

We next discuss cohomology with coefficients in $\so(3,1)$. These cohomology groups are classically studied and well understood. The following Lemma summarizes some well known properties of $H^1(\Gamma,\so(3,1)^{\rhohyp})$ that will be important for our purposes.

\begin{lemma}\label{so31cohomology}
Suppose that $M$ has $k$ cusps, then 
\begin{itemize}
	\item $H^1(\Delta,\so(3,1)^{\rhohyp})$ is $4k$-dimensional
	\item $H^1(\Gamma,\so(3,1)^{\rhohyp})$ is $2k$-dimensional
	\item $\res_\ast:H^1(\Gamma,\so(3,1)^{\rhohyp})\to H^1(\mu,\so(3,1)^{\rhohyp})$ is injective
\end{itemize}	
\end{lemma}

\begin{proof}
	The first point follows from Poincar\'e duality. More specifically, $H^0(\Delta,\so(3,1)^{\rhohyp})$ is by construction the sum of the infinitesimal centralizers of $\rhohyp(\Delta_i)$ in $\so(3,1)$. A simple computation shows that each of these is 2-dimensional and so $\dim H^0_{\rhohyp}(\Delta,\so(3,1)^{\rhohyp})=2k$. Since $\Delta:=\oplus_{i=1}^k\Delta_i$ is the sum of fundamental groups of a closed 2-dimensional manifolds Poincare duality implies that $\dim H^2_{\rhohyp}(\Delta,\so(3,1)^{\rhohyp})=2k$. Since the Euler characteristic of $\partial$ is zero it again follows from Poincare duality that $\dim H^1(\Delta,\so(3,1)^{\rhohyp})=4k$.
	
	The second and third points are algebraic consequences of the Thurston Dehn filling theorem and their proof can be found in \cite[Prop 3.3]{HePo} 
\end{proof}
 
 As a result of Lemma \ref{so31cohomology}, we see that the image of $H^1(\Gamma,\so(3,1)^{\rhohyp})$ is a half-dimensional subspace. This is not coincidental, as there turns out to be a symplectic form on $H^1(\Delta,\so(3,1)^{\rhohyp})$ induced by the cup product, for which the image of $H^1(\Gamma,\so(3,1)^{\rhohyp})$ is a Lagrangian subspace \cite[\S 5]{HePo}. 
 
 Cohomology with coefficients in $\v$ is less well understood, but in this setting we have the following weaker analogue of Lemma \ref{so31cohomology}, whose proof can be found in \cite{HePo}
 
\begin{lemma}[Cor.\ 5.2 \& Lem.\ 5.3 of \cite{HePo}]\label{vcohomology}
	Suppose that $M$ has $k$ cusps, then 
	\begin{itemize}
		\item $H^1(\Delta,\v^{\rhohyp})$ is $2k$-dimensional
		\item $\res_\ast(H^1(\Gamma,\v^{\rhohyp}))\subset H^1(\Delta,\v^{\rhohyp})$ is $k$-dimensional.
	\end{itemize}
\end{lemma}

We now have the requisite background and context to state our cohomological condition. A manifold $M$ is \emph{infinitesimally rigid rel.\ $\partial M$} if the map $\res_\ast:H^1(\Gamma,\g^{\rhohyp})\to H^1(\Delta,\g^{\rhohyp})$ is injective. To avoid cumbersome phrasing, we will often abbreviate this terminology and say that $M$ is infinitesimally rigid. In other words, $M$ is infinitesimally rigid rel.\ $\partial M$ if there are no infinitesimal deformations of $M$ that are infinitesimal conjugacies when restricted to each cusp.  This condition was first introduced in \cite{HePo}.  Some comments regarding this condition are in order. First, the map $\res_\ast:H^1(\Gamma,\so(3,1)^{\rhohyp})\to H^1(\Delta,\so(3,1)^{\rhohyp})$ factors through $\res_\ast H^1(\Gamma,\so(3,1)^{\rhohyp})\to H^1(\mu,\so(3,1)^{\rhohyp})$ and so by Lemma \ref{so31cohomology}, infinitesimal rigidity of $M$ is equivalent to the injectivity of $\res_\ast:H^1(\Gamma,\v^{\rhohyp})\to H^1(\Delta,\v^{\rhohyp})$. Second, by Lemma \ref{vcohomology}, the dimension of $H^1(\Gamma,v)$ is at least $k$ and so $M$ is infinitesimally rigid if the dimension of $H^1(\Gamma,\v^{\rhohyp})$ is \emph{exactly} $k$. 

There are infinitely many infinitesimally rigid cusped hyperbolic 3-manifolds. Specifically, Heusener and Porti \cite{HePo} show that infinitely many surgeries on the Whitehead link result in manifolds that are infinitesimally rigid. Examples of such families include infinitely many twist knots and infinitely many once-punctured torus bundles with tunnel number one. Furthermore, based on numerical computation by J.\ Danciger, G.-S. Lee, and the author it appears that infinitesimal rigidity is a fairly common property amongst 3-manifolds in the SnapPy \cite{SnapPy} cusped census.

On the other hand, there are also infinitely many cusped 3-manifolds that are not infinitesimally rigid. For example, if $M$ contains a closed, embedded, totally-geodesic hypersurface, then it is possible to perform a type of deformation called bending (see \cite{JohnMill} or \cite{BalMar} for details). These deformations are trivial when restricted to any cusp, and so if $M$ contains such a hypersurface then $M$ is not infinitesimally rigid rel.\ $\partial M$.

We close this section by describing an important consequence of infinitesimal rigidity. As we have seen, the set $H^1(\Gamma,\g^{\rhohyp})$ can be interpreted as non-trivial infinitesimal deformations of $\rhohyp$. Given a cohomology class, $[w]\in H^1(\Gamma,\g^{\rhohyp})$ one would like to know if there is a family $\rho_t:\Gamma\to G$ of representations that is tangent to $w$. In the language of algebraic geometry, $w$ is a tangent vector in the Zariski tangent space of the algebraic variety $\hom(\Gamma,G)$ and this question is equivalent to the question of whether or not $\rhohyp$ is a smooth point of $\hom(\Gamma,G)$. There are numerous examples where $\rhohyp$ fails to be a smooth point (see \cite{CoLoThist1} for explicit examples). There is also a related result of Kapovich--Millson \cite{KapMil} that, roughly speaking, says for 3-manifolds and representations into $\SL(2,\C)$ that arbitrary singularities are possible. However the following result from \cite{BDL} shows that for infinitesimally rigid 3-manifolds, $\rhohyp$ is a smooth point of $\hom(\Gamma,G)$ and $\rep(\Gamma,G)$. 

\begin{theorem}[see Thm 3.2 in \cite{BDL}]\label{infrigimpliessmooth}
	Suppose $M$ is a cusped finite volume hyperbolic 3-manifold with $\rhohyp:\Gamma\to SO(3,1)$ the holonomy of the complete hyperbolic structure on $M$. If $M$ is infinitesimally rigid rel.\ $\partial M$ then $\rhohyp$ is a smooth point of $\hom(\Gamma,G)$ and $[\rhohyp]$ is a smooth point of $\rep(\Gamma,G)$. 
	\end{theorem}
	
	For the sake of completeness, we include a proof of Theorem \ref{infrigimpliessmooth}, however the proof will be deferred until the next section since it requires the development of the appropriate obstruction theory for $\SL_4$ representations. 

\subsection{Obstruction theory}
In this section we let $G=\SL_4(\R)$. We now discuss the problem of when an infinitesimal deformation is tangent to a smooth path in $\hom(\Gamma,G)$. Roughly speaking, our strategy will be to start with $u\in Z^1(\Gamma,\g^{\rho_0})$ and try to construct a \emph{formal deformation of $\rho_0$ tangent to $u$} (i.e. a representation $\tilde\rho_t\in \hom(\Gamma,\SL_4(\R[[t]])$) whose ``formal derivative'' is $u$. If such a formal deformation can be constructed then we can apply a deep theorem of Artin \cite[Thm 1.2]{Art} to show that there is a smooth path $\rho_t$ in $\hom(\Gamma,G)$ tangent to $u$ at $\rho_0$. 

 Many of the results of this section are similar to those from \cite[\S 3]{HePoSuSL2} where the authors provide a detailed discussion of the analogous obstruction theory for $\SL_2(\C)$ representations (see Remark \ref{mistake}). The results in this section are stated for $\SL_4(\R)$, but as the reader can see, the arguments are general enough to apply to a wide variety of Lie groups.

Let $A_\infty=\R[[t]]$ be the $\R$-algebra of formal power series in one variable over $\R$. For each $k\geq 0$ define an $\R$-algebra $A_k=A_\infty/(t^{k+1})$, $G_k=\SL_4(A_k)$, and $\g_k$ to be the Lie algebra of $G_k$. In particular, $G_0=\SL_4(\R)$, $\g_0=\sl_4(\R)$ and $\g_k\cong \g_0\otimes A_k$. If $\rho_k:\Gamma\to G_k$ is a representation then combining $\rho_k$ and the adjoint representation turns $\g_k$ into a $\Gamma$-module which we refer to as $\g_k^{\rho_k}$

Let $k>l$, then there is a projection $\pi_{k,l}:A_k\to A_l$. When $l=0$, we denote $p_k:=\pi_{k,0}$. When $l=k-1$, we denote $\pi_{k-1}:=\pi_{k,k-1}$. By restricting to coefficients we also get maps on Lie groups and Lie algebras which we abusively denote by $\pi_{k,l}:G_k\to G_l$ and $\pi_{k,l}:\g_k\to \g_l$. 

When $l=0$, we get a (split) short exact sequence 
$$0\to G_k^0\to G_k\stackrel{p_k}{\to}G_0\to 0$$ 
which gives an identification $G_k\cong G_k^0\rtimes G_0$. The group $G_k^0$ is a Lie group with Lie algebra $\g_0\otimes \m_k$, where $\m_k=(t)$ is the unique maximal ideal of $A_k$. Since $\m_k^{k+1}=0$, this Lie algebra is nilpotent $\g_0\otimes \m_k$ can be turned into a Lie group using ``Campbell-Baker-Hausdorff'' multiplication, and this Lie group is isomorphic to $G_k^0$ (see \cite[\S 4]{GoMi} for more details).

We now discuss the case where $k=1$. Using the semi direct product structure on $G_1$ we see that any representation $\rho_1:\Gamma\to G_1$ can be written uniquely as $\rho_1(\gamma)=\exp(t u_1(\gamma))\rho_0(\gamma)$, where $\rho_0:\Gamma\to G_0$ is a homomorphism and $u_1:\Gamma\to \g_0$ is a function. In this setting we say that $\rho_1$ is an \emph{infinitesimal deformation of $\rho_0$}. The homomorphism condition for $\rho_1$ and the Campbell-Baker-Hausdorff formula combine to forces $u_1$ to satisfy the condition 
$$u_1(\gamma_1\gamma_2)=u_1(\gamma_1)+\rho_0(\gamma_1)\cdot u_1(\gamma_2).$$

\noindent On the other hand, it is straightforward to check that given $u\in Z^1(\Gamma,\g_0^{\rho_0})$, $\rho(\gamma)=\exp(t u(x))\rho_0(x)$ is an infinitesimal deformation. This construction gives a bijection between infinitesimal deformations of $\rho_0$ and 1-cocycles in $Z^1(\Gamma,\g_0^{\rho_0})$.

We partially generalize the previous phenomenon to show that representations into other $G_k$ are also related to cocycles in group cohomology. Let $u=\sum_{i=0}^\infty c_i t^i\in C^1(\Gamma,\g_\infty)$, where each $c_i\in C^1(\Gamma,\g_0)$. We also define the \emph{formal derivative} of $u$, which we denote by $u':=\sum_{i=1}^\infty ic_it^{i-1}\in C^1(\Gamma,\g_\infty)$. For any such cochain we can also define  $u_k=\sum_{i=1}^k c_it^i\in C^1(\Gamma,\g_k)$. Next, we can define 
$$du=u'+\frac{1}{2!}[u,u']+\frac{1}{3!}[u,[u,u']]+\ldots=\sum_{i=0}^\infty \frac{1}{(i+1)!}\ad_{u}^i(u').$$

 Regarding $du$ as a power series, it can be rewritten as $du=\frac{\exp(\ad_u)-1}{\ad_u}\left(u'\right)$. The importance of $du$ is that it appears in the formula for differentiating the exponential map. More precisely, if $u\in C^1(\Gamma,\g_\infty)$ has trivial constant term (i.e $c_0=0$) then $\frac{d}{dt}(\exp(u(\gamma)))=du(\gamma)\exp(u(\gamma))$, for any $\gamma\in \Gamma$. In what follows, we will omit $\gamma$ from our notation and just use the expression $\frac{d}{dt}(\exp(u))=du\exp(u)$. 
 
 

The map $du$ also gives rise to maps $D_k:G_{k+1}^0\to \g_k$ for each $k$, defined as follows: each $g_{k}\in G_{k+1}^0$ can be written uniquely was $\exp(u_{k})$ for some $u_{k}\in \g_{0}\otimes \m_{k+1}$. If we let $u\in \g_0\otimes \m_\infty$ be such that $\pi_{\infty,k+1}(u)=u_{k}$ then we define $D_k(g_{k})=\pi_{\infty,k}(du)$. This is well defined since if $w$ is another element projecting to $u_{k}$ then $w=u+v$, where $v\in \g_0\otimes (t^{k+2})$. It follows that $d(u+v)=u'+v'+\sum_{i=1}^\infty \frac{1}{(i+1)!}\ad_{u+v}^i(u'+v')$. Under $\pi_{\infty, k}$, $u'$ maps to $u'_{k}$, $v'$ maps to 0, and $\sum_{i=1}^\infty \frac{1}{(i+1)!}\ad_{u+v}(u'+v')$ maps to $\sum_{i=1}^k\frac{1}{(i+1)!}\ad^i_{u_{k}}(u'_{k})$. It follows that $\pi_{\infty,k}(du)=\pi_{\infty,k}(dw)$.

\begin{lemma}\label{kderivative}
	The map $D_k$ is a derivation in the sense that it satisfies the formula 
	$$D_k(\exp(u_k)\exp(v_k))=D_k(\exp(u_k))+\exp(u_k)\cdot D_k(\exp(v_k))$$
	for every $u_k,v_k\in \g_0\otimes \m_{k+1}$. 
\end{lemma}

\begin{proof}
	Let $u,v\in \g_0\otimes \m_\infty$ be elements mapping to $u_k$ and $v_k$ under $\pi_{\infty,k+1}$, let $z\in \g_0\otimes \m_\infty$ such that $\exp(z)=\exp(u)\exp(v)$ and let $z_k=\pi_{\infty,k+1}(z)$. Note that $\exp(z_k)=\exp(u_k)\exp(v_k)$. Using the product rule for differentiation we find that 
	$$\frac{d}{dt}(\exp(u)\exp(v))=du\exp(u)\exp(v)+\exp(u)dv\exp(v)=(du+\exp(u)\cdot dv)\exp(u)\exp(v)$$
	
	On the other hand we see that 
	$$\frac{d}{dt}(\exp(z))=dz\exp(z)$$
	
	By construction, these two expressions are equal and so we find that $dz=du+\exp(u)\cdot dv$. It then follows that 
	\begin{align}
		D_k(\exp(u_k)\exp(v_k))=D_k(\exp(z_k))&=\pi_{\infty,k}(dz)\nonumber\\
		=\pi_{\infty,k}(du+\exp(u)\cdot dv)=\pi_{\infty,k}(du)+\exp(u_k)\cdot \pi_{\infty,k}(dv)&=D_k(u_k)+\exp(u_k)\cdot D_k(v_k)\nonumber
	\end{align}

\end{proof}



The following lemma shows how representations into $G_k$ give rise to cocycles. 
 
\begin{lemma}\label{reptococycle}
Let $u\in C^1(\Gamma,\g_\infty)$ be a cochain with trivial constant term and let $\rho_0:\Gamma\to G_0$ be a homomorphism. If  $\rho_k:\Gamma\to G_k$ is a homomorphism of the form $\rho_k(\gamma)=\exp(u_k(\gamma))\rho_0(\gamma)$ and $\rho_{k-1}=\pi_{k-1}\circ \rho_k$, then $du_k:=\pi_{\infty,k-1}(du)\in Z^1(\Gamma,\g_{k-1}^{\rho_{k-1}})$. 	
\end{lemma}

\begin{proof}

Since $\rho_k$ is a homomorphism $\rho_k(\gamma_1\gamma_2)=\rho_k(\gamma_1)\rho_k(\gamma_2)$ and so it follows that $\exp(u_k(\gamma_1\gamma_2))=\exp(u_k(\gamma_1))\exp(\rho_0(\gamma_1)\cdot u_k(\gamma_2))$. Recalling that $du_k:=\pi_{\infty,k-1}(du)$ and applying $D_{k-1}$ to the previous equality we find that 
\begin{align*}
	du_k(\gamma_1\gamma_2)&=du_k(\gamma_1)+\exp(u_k(\gamma_1))\cdot (\rho_0(\gamma_1)\cdot du_k(\gamma_2))\\
	&=du_k(\gamma_1)+\rho_{k-1}(\gamma_1)\cdot du_k(\gamma_2)
\end{align*}

\noindent It follows that $du_k\in Z^1(\Gamma,\g_{k-1}^{\rho_{k-1}})$. 
\end{proof}

\begin{remark}\label{mistake}
	Lemma \ref{reptococycle} should be compared with Lemma 3.3 of \cite{HePoSuSL2}. However, there is a small mistake in the statement of the theorem caused by the authors having the wrong formula for $du$. After a small change in their formula the theorem and its proof are correct. The author would like to thank Joan Porti for providing him with the formula for $du$ and explaining its relationship to differentiation of exponentials. 
\end{remark}

Next, suppose that $\rho_k:\Gamma\to G_k$ is a homomorphism, $\rho_{k-1}=\pi_{k-1}\circ \rho_k$ and $\rho_0=p_k\circ \rho_k$ and observe that there is a short exact sequence of $\Gamma$-modules 
$$0\to \g_0^{\rho_0}\stackrel{\alpha_k}{\longrightarrow}\g_{k}^{\rho_k}\stackrel{\pi_{k-1}}{\longrightarrow}\g_{k-1}^{\rho_{k-1}}\to 0,$$
where $\alpha_k$ is induced by the map from $A_0\to A_k$ given by $x\mapsto t^kx$. This short exact sequence induces a long exact sequence on cohomology, a portion of which is 
\begin{align}\label{exactsequence}
H^1(\Gamma,\g_{k}^{\rho_k})\stackrel{\pi_{k-1}}{\longrightarrow}H^1(\Gamma,\g_{k-1}^{
\rho_{k-1}})\stackrel{\delta}{\longrightarrow}H^2(\Gamma,\g_0^{\rho_0})
\end{align}

We now have the tools to address the following problem. Suppose that $\rho_k:\Gamma\to G_k$ is a homomorphism and $\rho_0=p_k\circ \rho_k$. When is there a homomorphism $\rho_{k+1}:\Gamma\to G_{k+1}$ such that $\rho_k=\pi_k\circ \rho_{k+1}$? The following theorem provides the answer in the form of an obstruction (see also \cite[Prop.\ 3.1]{HePoSuSL2}). 

\begin{theorem}\label{obstructions}
	Given, $u$, $\rho_0$ and $\rho_k$ as above there is a cohomology class $o(\rho_k)\in H^2(\Gamma,\g_0^{\rho_0})$ with the property that there is a homomorphism $\rho_{k+1}:\Gamma\to G_{k+1}$ so that $\rho_k=\pi_k\circ \rho_{k+1}$ if and only if $o(\rho_k)=0$. 
	
	Furthermore the obstructions are natural in the sense that if $\Gamma'$ is another group, $f:\Gamma'\to \Gamma$ is a homomorphism, and $\rho_k'=f\circ \rho_k$ then $o(\rho_k')=f^\ast(o(\rho_k))$. 
	
	\end{theorem} 
 
 \begin{proof}
 	Write $\rho_k(\gamma)=\exp(u_k(\gamma))\rho_0(\gamma)$. By Lemma \ref{reptococycle} $[du_k]$ is an element of $H^1(\Gamma,\g_{k-1}^{\rho_{k-1}})$ and we define $o(\rho_k)=\delta([du_k])$. 
 	
 	First, suppose that there is a homomorphism $\rho_{k+1}:\Gamma\to G_{k+1}$ so that $\rho_k=\pi_k\circ \rho_{k+1}$. Write $\rho_{k+1}(\gamma)=\exp(u_{k+1}(\gamma))\rho_0(\gamma)$ where $u_{k+1}\in C^1(\Gamma,\g_{k+1})$. By Lemma \ref{reptococycle} we see that $[du_{k+1}]\in H^1(\Gamma,\g_{k}^{\rho_k})$. Furthermore, since $\rho_{k}=\pi_k\circ \rho_{k+1}$ it follows that $\pi_{k-1}[du_{k+1}]=[du_k]$. From exactness of \eqref{exactsequence} it follows that $o(\rho_k)=0$. 
 	
 	On the other hand, suppose that $o(\rho_k)=0$. Again using exactness of \eqref{exactsequence} we see that there is a class $[v]\in H^1(\Gamma,\g_{k}^{\rho_k})$ so that $\pi_{k-1}([v])=[du_k]$. Let $v\in Z^1(\Gamma,\g_k^{\rho_k})$ be a representative of this class, then $\sigma(\gamma)=\exp(s v(\gamma))\rho_k(\gamma)$ defines a homomorphism from $\Gamma$ to $\SL_4(B_k)$, where $B_k=A_k[[s]]/(s^2)$. There is a homomorphism, $g:B_k\to A_{k+1}$ given by $g(a+sb)=a+t^{k+1}b$ and this map induces a homomorphism $g:\SL_4(B_k)\to G_{k+1}$. Since $\sigma$ is a homomorphism, so is $\tilde \sigma=g\circ \sigma$. Furthermore, 
 	$$\tilde \sigma(\gamma)=\exp(t^{k+1} v(\gamma))\rho_k(\gamma)=\exp(t^{k+1}v(\gamma))\exp(u_k(\gamma))\rho_0(\gamma).$$ 
 	
 	
 	\noindent and so $\rho_k=\pi_k\circ \tilde \sigma$. 
 	
 	The last statement follows from the naturality of the long exact sequence in cohomology \eqref{exactsequence} (see \cite[Ch.\ III]{BrownCohom}. 
 \end{proof}
 
 We will need the following two lemmas from \cite{BDL} in order to prove Theorem \ref{infrigimpliessmooth}. 
 
 \begin{lemma}[Lem 3.6 in \cite{BDL}]\label{H2injective}
 Let $M$ be a finite volume hyperbolic 3-manifold with boundary $\partial=\sqcup_{i=1}^k\partial_i$ and $\Delta=\oplus_{i=1}^k\pi_1(\partial_i)$. Suppose that $M$ is infinitesimally rigid, then the map 
 $$\res_\ast:H^2(\Gamma,\g^{\rhohyp})\to H^2(\Delta,\g^{\rhohyp})$$
 is injective.	
 \end{lemma}
 
 \begin{lemma}[Lem 3.4 in \cite{BDL}]\label{boundarysmoothpoint}
 	Let $M$ be a finite volume hyperbolic 3-manifold with boundary $\partial=\sqcup_{i=1}^k\partial_i$ and $\Delta=\oplus_{i=1}^k\pi_1(\partial_i)$ then $\res(\rhohyp)$ is a smooth point of $\hom(\Delta,G)$. 
 \end{lemma}
 
 \begin{proof}[Proof of Theorem 2.3]
 	If $\rhohyp$ is a smooth point of $\hom(\Gamma,G)$ then $[\rhohyp]$ is a smooth point of $\rep(\Gamma,G)$. To see this observe that the representation $\rhohyp$ is well known to be irreducible. Irreducibility is an open condition, and so there is a neighborhood $U\subset \hom(\Gamma,G)$ consisting of irreducible representations. Since $\rhohyp$ is a smooth point we can assume that $U$ is a manifold. Let $V=G\cdot U$ be the orbit of this set, then $V$ is also a manifold and there is a simply transitive, and hence free and proper, action of $G$ on $V$. It follows that $V/G\subset \rep(\Gamma,G)$ is a manifold that contains $[\rhohyp]$ and so $[\rhohyp]$ is a smooth point of $\rep(\Gamma,G)$. 
 	
 	Thus it suffices to show that $\rho_0:=\rhohyp$ is a smooth point of $\hom(\Gamma,G)$. Let $u\in Z^1(\Gamma,\g^{\rho_0})$ be an infinitesimal deformation. Such an infinitesimal deformation gives rise to a representation $\rho_1:\Gamma\to G_1$ lifting $\rho_0$. Suppose for contradiction that $\rho_1$ cannot be lifted to a representation $\rho_\infty:\Gamma\to G_\infty$, then there must be some $k$ and a representation $\rho_k:\Gamma\to G_k$ lifting $\rho_1$ that cannot be lifted to a representation $\rho_{k+1}:\Gamma\to G_{k+1}$. By Theorem \ref{obstructions} this implies that $o_k:=o(\rho_k)\neq 0\in H^2(\Gamma,\g^{\rho_0})$. Next, let $o_k'=\res_\ast(o_k)\in H^2(\Delta,\g^{\rho_0})$. By the naturality of the obstruction (see Theorem \ref{obstructions}) it follows that $o_k'$ is the obstruction to lifting $\res(\rho_k)$ to $G_{k+1}$. Furthermore, $o_k\neq 0$ and so Lemma \ref{H2injective} implies that $o_k'\neq 0$. However, by Lemma \ref{boundarysmoothpoint} $\res(\rho_0)$ is a smooth point of $\hom(\Delta,G)$ and so it is possible to lift $\res(\rho_1)$ to $G_{k+1}$, and hence $o'_k=0$, which is a contradiction. It follows that $\rho_1$ can be lifted to $\rho_\infty:\Gamma\to G_\infty$. 
 	
 	Finally, by applying Theorem 1.2 of \cite{Art} it follows that we can find a curve of representations $\rho_t$ in $\hom(\Gamma,G)$ that is tangent to $u$ at $\rho_0$.

 	  \end{proof}

\section{The Slice}\label{slicesection}
Let $G=\SLpm(4,\R)$ and let $G_A$ the group of affine transformations of $\R^3$, both of which can be thought of as a subgroups of $\GL(4,\R)$ and let $\g$ and $\g_a$ be the corresponding Lie algebras. There is a natural injective map from $\varpi:G_A\to G$ given by $M\mapsto \abs{\det(M)}^{-1/4}$. The corresponding map  $\varpi:\g_a \to \g$ at the level of Lie algebras is given by $v\mapsto v-\frac{\tr(v)}{4}\operatorname{Id}$. If $\Gamma$ is a finitely generated group then the above injection induces an injection from $\hom(\Gamma,G_A)$ into $\hom(\Gamma,G)$.
\\\indent
Let 
$$S=\{(a,b,x_1,y_1,x_2,y_2)\in \R^6\mid y_1x_2-x_1y_2= \pm1\}$$
It is a simple exercise in differential topology to see that $S$ is a smooth 5-dimensional manifold.
\\\indent
Next, let $C=\{(a,b,x_1,y_1,x_2,y_2)\in S\mid a=b=0\}$. $C$ is a smooth 3-dimensional submanifold of $S$ and there is a function $\CS:C\to \C$ called the \emph{cusp shape function}, given by $(0,0,x_1,y_1,x_2,y_2)\mapsto \frac{x_1+iy_1}{x_2+iy_2}$. 
\\\indent
Here is some information about the calculus of $\CS$. 
\begin{lemma}\label{csderiv}
	$\CS$ is a surjective submersion of $C$ onto $\C\bs \R$. Consequently, $\{CS^{-1}(z)\mid z\in \C\bs \R\}$ gives a foliation of $C$ by smooth 1-manifolds.  
\end{lemma}

\begin{proof}
First, Let $f>0$, then $\CS(e/\sqrt{f},\pm\sqrt{f},1/\sqrt{f},0)=e\pm if$, and so $\C\bs \R$ is contained in the image of $\CS$. Furthermore, since $y_1x_2-x_1y_2=\pm 1$ it follows that $x_1+iy_1$ and $x_2+iy_2$ are linearly independent over $\R$ and hence $\CS(x_1,y_1,x_2,y_2)=\frac{x_1+iy_1}{x_2+iy_2}\in \C\bs \R$. 

Next, identify $\C$ with $\R^2$ in the usual way and identify $C$ with a subset of $\R^4$. If $h:\R^4\to \R$ is given by $(x_1,y_1,x_2,y_2)\mapsto y_1x_2-x_1y_2$ and $v\in C$ then $T_vC=\ker d h(v)$.

 Let $f,g:\R^4\to \R$ be given by $f(x_1,y_1,x_2,y_2)=x_2^2+y_2^2$ and $g(x_1,y_1,x_2,y_2)=x_1x_2+y_1y_2$.  Then we can write $\CS=F_1\circ F_2$ where $F_2:C\to \R^2$ is given by $v\mapsto (f(v),g(v))$ and $F_1:\R^2\to \R^2$ is given by $(f,g)\mapsto \left(\frac{g}{f},\frac{1}{f}\right)$. Let $v\in C$ and $w\in T_vC$ be a tangent vector, then using the chain rule we find that 
$$\CS_\ast(w)=\left(\frac{(f(v)d g(v)-g(v)d f(v))(w)}{f(v)^2},\frac{-d f(v)(w)}{f(v)^2}\right),$$
thus the kernel of $\CS_\ast$ is equal to $\ker d f(v)\cap \ker d g(v)\cap \ker d h(v)$. It is easy to check that $d f(v)$, $d g(v)$, and $d h(v)$ are linearly independent 1-forms for all $v\in C$ and so the kernel of $\CS_\ast$ is 1-dimensional. It follows that $\CS$ is a submersion at $v$ and hence a submersion since $v$ was arbitrary. 
\end{proof}

We now set some notation. In what follows we follow the convention that standard fonts are used for spaces of parameters and calligraphic fonts are used to denote the spaces being parameterized. We now use $S$ to parameterize families of representations of $\Z^2$ into $G_A$ and $G$. Fix once and for all a generating set $\{\gamma_1,\gamma_2\}$ for $\Z^2$. For each $s=(a,b,x_1,y_1,x_2,y_2)\in S$ we can define a representation $\rho_s:\Z^2\to G_A$ via $\rho_s(\gamma_i)=\exp(m_s^i)$, where for $i\in \{1,2\}$
\begin{equation}\label{slicerep}
	m_s^i=\begin{pmatrix}
		0 & x_i & y_i & 0\\
		0 & a x_i & 0 & x_i\\
		0 & 0 & b y_i & y_i\\
		0 & 0 & 0 & 0
	\end{pmatrix}
\end{equation}

By examining the entries of \eqref{slicerep} it is easy to see that $F:S\to \hom(\Z^1,G_A)$ given by $s\mapsto \rho_s$ is an injective immersion of $S$ into $\hom(\Z^2,G)$ whose image, which we denote $\mathcal{S}_A$, is an embedded submanifold. Let $\mathcal{C}_A$ be the submanifold of $\mathcal{S}_A$ corresponding to $C$.


 There is another map $\tilde F:S\to \hom(\Z^2,G)$ given by $s\mapsto \tilde \rho_s$, where $\tilde\rho_s=\varpi\circ \rho_s$ and we denote the images of $S$ and $C$ under $\tilde F$ by $\Sc$ and $\Scusp$, respectively.  It is easy to see that $\mathcal{S}_A$ and $\Sc$ (resp.\ $\mathcal{C}_A$ and $\Scusp$) are diffeomorphic via $\rho_s\mapsto \tilde \rho_s$. If $\rho_s\in \Sc$ (or $\Sca$) then we call $s\in S$ the \emph{coordinates of $\rho_s$}.  The reason for using $\Sc$ (as opposed to $\Sca$) is that the transversality argument in Section \ref{transargument} takes place in $\SL(4,\R)$ and not $\GL(4,\R)$.

In terms of hyperbolic geometry, $\CS$ gives the cusp shape of the representation $\rho_s$ (with respect to the generating set $\{\gamma_1,\gamma_2\}$. It is well known that if $s,s'\in C$ then $\rho_s$ (resp.\ $\tilde \rho_s)$ is conjugate to $\rho_{s'}$ (resp.\ $\tilde\rho'_s$) in $G_A$ (resp.\ $G$) if and only if $\CS(s)=\CS(s')$. As a result, $S$ {\bf does not} give a parameterization of the image of $\Sc$ in $\hom(\Z^2,G)/G$ since there is redundancy coming from representations with the same cusp shape. However, we will see shortly, this is the only redundancy that arises when projecting $\Sc$ to $\hom(\Z^2,G)/G$ near $\Scusp$.

There is another way of viewing the above construction that is also useful: let 
$$x_{a,b}=\begin{pmatrix}
	0 & 1 & 0 & 0\\
	0 & a & 0 & 1\\
	0 & 0 & 0 & 0\\
	0 & 0 & 0 & 0
\end{pmatrix}\quad\ 
y_{a,b}=\begin{pmatrix}
	0 & 0 & 1 & 0\\
	0 & 0 & 0 & 0\\
	0 & 0 & b & 1\\
	0 & 0 & 0 & 0
\end{pmatrix}$$

We can view $\{x_{a,b},y_{a,b}\}$ as a basis of an abelian Lie algebra $\mathfrak{a}_{a,b}$ of an abelian Lie subgroup $A_{a,b}\subset G_A$ isomorphic to $\R^2$. If $s=(a,b,x_1,y_1,x_2,y_2)$ then the representation $\rho_s$ has image in $A_{a,b}$. Furthermore, if we let $v_1=(x_1,y_1)$ and $v_2=(x_2,y_2)$, then the defining condition for $S$ is a ``determinant'' condition that ensures that $v_1$ and $v_2$ are linearly independent vectors in $\R^2$, and so we see that the image of $\rho_s$ is always a lattice in $A_{a,b}$. The group $A_{0,0}$ is equal to $T(0)$, and so we immediately see that many representations in $\Scusp$ are holonomies of type 0 generalized cusps. The following Theorem shows that the remaining representations in $\Sc$ are also holonomies of generalized cusps.   

\begin{theorem}\label{sliceisgencuspholonomy}
	Let $\rho\in \Sc$, then $\rho$ is the holonomy of a  generalized cusp of type 0, type 1, or type 2. 
\end{theorem}

\begin{proof}

If $\rho\in \Scusp$ then the image of $\rho$ is a lattice in $T(0)$ and so $\rho$ is the holonomy of a type 0 generalized cusp. On the other hand, suppose that $s=(a,b,x_1,y_1,x_2,y_2)\in S$ is such that $(a,b)\neq(0,0)$. There are two cases: either $a$ or $b$ (but not both) is zero or both $a$ and $b$ are non-zero. We begin with the first case. By performing a conjugacy that permutes the second and third coordinates, if necessary, we can assume without loss of generality that $b=0$. Let
$$N_{a,b}(x,y)=\exp\begin{pmatrix}
	0 & x & y & 0\\
	0 & a x & 0 & x\\
	0 & 0 & b y & y\\
	0 & 0 & 0 & 0
\end{pmatrix}$$
We regard $N_{a,b}(x,y)$ as an arbitrary element of the Lie group $A_{a,b}$. Next, let $P_{(12)}$ be the $4\times 4$ matrix that permutes the first two coordinates, let 
$$C_a=\begin{pmatrix}
	1 & -1/a & 0 & 0\\
	0 & a & 0 & 1\\
	0 & 0 & 1 & 0\\
	0 & 0 & 0 & 1
\end{pmatrix},$$
and let $\tilde C_a=P_{(12)}C_a$.  Observe that 
$$\tilde C_aN_{a,0}(x,y)\tilde C_a^{-1}=M_1(ax, y,-a^{-1}x).$$
As a result we see that $A_{a,0}$ is conjugate to $T(a)$ and that the image of $\tilde C_a\cdot \rho$ is a lattice in $T(a)$. It follows that $\rho$ is the holonomy of a type 1 generalized cusp. 

In the case where both $a$ and $b$ are non-zero let 
$$D_{a,b}=\begin{pmatrix}
	1 & -1/a & -1/b & 0\\
	0 & a & 0 & 1\\
	0 & 0 & b & 1\\
	0 & 0 & 0 & 1
\end{pmatrix},$$
let $P_{(123)}$ be the $4\times 4$ matrix that cyclically permutes the first 3 coordinates, and let $\tilde D_{a,b}=P_{(123)}D_{a,b}$. In this case we find that 
$$\tilde D_{a,b}N_{a,b}(x,y)\tilde D_{a,b}^{-1}=M_2(a x,by,-a^{-1}x-b^{-1}y),$$
and thus $A_{a,b}$ is conjugate to $T(a,b)$. Arguing as before this implies that $\tilde D_{a,b}\cdot \rho$ is the holonomy of a type 2 generalized cusp.

\end{proof}

We now describe the tangent spaces for $\Sc$. Since $\Sc$ is a subvariety of $\hom(\Z^2,G)$, its tangent space is naturally a subspace of the Zariski tangent space of $\hom(\Z^2,G)$. For simplicity of notation we will denote $T_{\rho_s}\Sc$ by $T_s\Sc$. The tangent bundle $TS$ is spanned by 5 vector fields each of which can be written as a linear combination of the vector fields $\frac{\partial}{\partial a}$, $\frac{\partial}{\partial b}$, $\frac{\partial}{\partial x_1}$, $\frac{\partial}{\partial y_1}$, $\frac{\partial}{\partial x_2}$, and $\frac{\partial}{\partial y_2}$. Using $\tilde F$ we can push these vector fields on $T\Sc$, which by abuse of notation we give the same names. Again, these vector fields pointwise span $T\Sc$.

 Recall from Section \ref{infdefs} that $T_\rho \hom(\Z^2,G)$ can be identified with the space $Z^1(\Z^2,\g^{\rho})$ of 1-cocycles with coefficients in $\g$ twisted by $\Ad(\rho)$. It is possible to regard the elements of $\left\{\frac{\partial}{\partial a},\frac{\partial}{\partial b},\frac{\partial}{\partial x_1},\frac{\partial}{\partial y_1},\frac{\partial}{\partial x_2},\frac{\partial}{\partial y_2}\right\}$ as 1-cocycles in $Z^1(\Z^2,\g^{\rho_s})$.  When $s\in C$ this procedure can be made quite explicit by describing how each partial derivative acts on a generating set. Recall that $\{\gamma_1,\gamma_2\}$ is a fixed choice of generating set for $\Z^2$. Before proceeding, we need the following Lemma
 
 \begin{lemma}\label{vcoboundary}
 	Let $s\in C$ and let $u_1,u_2\in \R$. Then there is a cocycle $z\in Z^1(\Z^2,\v^{\rho_s})$ with the property that 
 	$$z(\gamma_i)=\begin{pmatrix}
 		0 & 0 & 0 & u_i\\
 		0 & 0 & 0 & 0\\
 		0 & 0 & 0 & 0\\
 		0 & 0 & 0 & 0
 	\end{pmatrix}.$$
 	Furthermore, this cocycle is a coboundary. 
 \end{lemma}
\begin{proof}
It is easy to check that for any $v_1,v_2\in \R$ that 
$$v=\begin{pmatrix}
0 & v_1 & v_2 & 0\\
0 & 0 & 0 & -v_1\\
0 & 0 & 0 & -v_2\\
0 & 0 & 0 & 0 	
\end{pmatrix}
$$	
is an element of $\v$. Since $s\in C$ we can write $s=(0,0,x_1,y_1,x_2,y_2)$, and so there is a coboundary in $B^1(\Z^2,\v^{\rho_s})$ that maps $\gamma_i$ to 
$$v-\rho_s(\gamma_i)\cdot v=\begin{pmatrix}
	0 & 0 & 0 & 2(v_1x_i+v_2y_i)\\
	0 & 0 & 0 & 0\\
	0 & 0 & 0 & 0\\
	0 & 0 & 0 & 0
	\end{pmatrix}. $$
	Since $s\in S$ it is possible to find $v_1$ and $v_2$ that satisfy the equations
	\begin{equation*}
	\begin{matrix}
		2v_1x_1+2v_2y_1=u_1\\
		2v_1x_2+2v_2y_2=u_2
	\end{matrix}	
	\end{equation*}
Thus there is a cocycle with the required properties and this cocycle is a coboundary. 
\end{proof}

In order to compute the cocycles, we observe that $\rho_s(\gamma_i)=\exp(m_s^i),$
where $m_s^i$ is the matrix from \eqref{slicerep}. Note that if $s\in  C$ then $m_s^i$ is  nilpotent (its third power is 0). This is convenient since when one writes $\exp(m_s^i)$ as a power series and and takes partial derivatives there will be only finitely many non-zero terms.  More specifically, let $z\in\{a,b,x_1,y_1,x_2,y_2\}$ and $D:=D_z^i$ be the derivative of $m:=m_s^i$ with respect $z$ at $s$ then the partial derivative of $\rho_s(\gamma_i)$ with respect to $z$ at $s$ can be computed by differentiating the power series defining $\rho_s(\gamma_i)$ term by term. This procedure results in the following formula: 

\begin{equation}\label{derivFormula}
	D+\frac{1}{2}\left(D m+m D\right)+\frac{1}{6}\left(m^2D+mDm+Dm^2\right)+\frac{1}{24}\left(m^2 D m+m D m^2\right)+\frac{1}{120} m^2 D m^2.
\end{equation}
Here, all other terms vanish since $m^3=0$. 
	
We can now compute the relevant cocycles. We begin with $\frac{\partial }{\partial x_1}$. Since $\rho_s(\gamma_2)$ is independent of $x_1$ it follows that $\frac{\partial}{\partial x_1}(\gamma_2)=0$. Next, by combining equations \eqref{cocycleFormula} and \eqref{derivFormula}, it follows that $\frac{\partial }{\partial x_1}(\gamma_1)=\xi_1$, where

$$\xi_1=\begin{pmatrix}
	0 & 1 & 0 & 0\\
	0 & 0 & 0 & 1\\
	0 & 0 & 0 & 0\\
	0 & 0 & 0 & 0
\end{pmatrix}
.$$

 The element $\xi_1$ is easily checked to be in $\so(3,1)$ and so it follows that the image $\left[\frac{\partial}{\partial x_1}\right]$ of $\frac{\partial}{\partial x_1}$ in $H^1(\Z^2,\g^{\rho_s})$ is contained in $H^1(\Z^2,\so(3,1)^{\rho_s})$. Similar computations shows that $\left[\frac{\partial}{\partial y_1}\right]$, $\left[\frac{\partial}{\partial x_2}\right]$, and $\left[\frac{\partial}{\partial y_2}\right]$ are also contained in $H^1(\Z^2,\so(3,1)^{\rho_s})$.

Next, using equations \eqref{cocycleFormula} and \eqref{derivFormula}, it follows that $\frac{\partial}{\partial a}(\gamma_i)=\alpha_i$ where
\begin{equation}\label{alphamat}
\alpha_i=\begin{pmatrix}
	-\frac{x_i}{4} & \frac{x_i^2}{4} & 0 & 0\\
	0 & \frac{3 x_i}{4} & 0 & -\frac{x_i^2}{4}\\
	0 & 0 & -\frac{x_i}{4} & 0\\
	0 & 0 & 0 & -\frac{x_i}{4}
\end{pmatrix}+
\begin{pmatrix}
	0 & 0 & 0 & -\frac{x_i^3}{3}\\
	0 & 0 & 0 & 0\\
	0 & 0 & 0 & 0\\
	0 & 0 & 0 & 0
\end{pmatrix}
\end{equation}

A similar computation shows that $\frac{\partial}{\partial b}(\gamma_i)=\beta_i$ where
\begin{equation}\label{betamat}
\beta_i=\begin{pmatrix}
	-\frac{y_i}{4} & 0 &  \frac{y_i^2}{4} & 0\\
	0 & -\frac{ y_i}{4} & 0 & 0\\
	0 & 0 & \frac{3y_i}{4} & -\frac{y_i^2}{4}\\
	0 & 0 & 0 & -\frac{y_i}{4}
\end{pmatrix}+
\begin{pmatrix}
	0 & 0 & 0 & -\frac{y_i^3}{3}\\
	0 & 0 & 0 & 0\\
	0 & 0 & 0 & 0\\
	0 & 0 & 0 & 0
\end{pmatrix}
\end{equation}

It follows that both $\frac{\partial}{\partial a}$ and $\frac{\partial}{\partial b}$ are elements of $Z^1(\Z^2,\v^{\rho_s})$. Note that the in the formulas for \eqref{alphamat} and \eqref{betamat} that the second terms give rise to a cocycle (see Lemma \ref{vcoboundary}), and so when passing to cohomology, we can simplify our formulas. More precisely, there are coboundaries $z_a$ and $z_b$ such that  $z_a(\gamma_i)=\frac{x_i^3}{3}e_1\otimes e_4^\ast$ and $z_b(\gamma_i)=\frac{y_i^3}{3}e_1\otimes e_4^\ast$,  (here $\{e_1,\ldots,e_4\}$ is the standard basis for $\R^4$). Next, define $D_a=\frac{\partial}{\partial a}+z_a$ and $D_b=\frac{\partial}{\partial b}+z_b$. Since $z_a$ and $z_b$ are coboundaries we see that $\left[\frac{\partial}{\partial a}\right]=[D_a]$ and $\left[\frac{\partial}{\partial b}\right]=[D_b]$, however the formulas for $D_a$ and $D_b$ are simpler. The next result shows that $\{[D_a],[D_b]\}$ is a basis for $H^1(\Z^2,\v^{\rho_s})$.


\begin{lemma}\label{vbasis}
	Let $s\in C$, then $\{[D_a],[D_b]\}$ is a basis for $H^1(\Z^2,\v^{\rho_s})$. 
\end{lemma}

\begin{proof}
	From Lemma \ref{vcohomology} it follows that $H^1(\Z^2,\v^{\rho_s})$ is 2-dimensional, so it suffices to show that $[D_a]$ and $[D_b]$ are linearly independent.  Let $s=(0,0,x_1,y_1,x_2,y_2)$ and suppose that there are $c_a$, $c_b$ and $u\in \v$ so that for any $\gamma\in \Z^2$, 
\begin{equation}\label{vbasiseq}
	c_aD_a(\gamma)+c_bD_b(\gamma)=u-\rho_s(\gamma)\cdot u
\end{equation}
  An arbitrary element of $u\in \v$ is of the form 
	$$u=\begin{pmatrix}
		-\frac{u_5+u_8}{2} & u_1 & u_2 & u_3\\
		u_4 & u_5 & u_6 & -u_1\\
		u_7 & u_6 & u_8 & -u_2\\
		u_9 & -u_4 & -u_7 & -\frac{u_5+u_8}{2}
	\end{pmatrix}$$
	and so $u-\rho_s(\gamma_i)\cdot u$ is
	$$\begin{pmatrix}
		-u_4x_i-u_7y_i-\frac{1}{2}u_{9}(x_i^2+y_i^2) & \ast & \ast & \ast\\
		-u_{9}x_i & x_i(2u_4+u_9x_i) & u_7x_i+(u_4+u_9x_i)y_i & \ast\\
		-u_9x_i & u_7x_i+(u_4+u_9x_i)y_i & y_i(2u_7 +u_9y_i) & \ast\\
		\ast & \ast & \ast & -u_4x_i-u_7y_i-\frac{1}{2}u_{9}(x_i^2+y_i^2)
	\end{pmatrix}$$
	The image of both $D_a$ and $D_b$ consist entirely of upper triangular elements of $\v$. It follows that the only way that \eqref{vbasiseq} can be satisfied is if $u_4=u_7=u_9=0$, and hence $c_a=c_b=0$
\end{proof}

Using the above description allows us to prove some useful intersection properties of $T_s\Sc$ when $s\in \Scusp$. 
\begin{proposition}\label{cuspshapetang}
	Let $s\in C$ and let $V$ and $W$ be the images of $T_s\Sc$ and $T_s\Scusp$ in $H^1(\Z^2,\g^{\rho_s})$, respectively. Then 
	$$H^1(\Z^2,\so(3,1)^{\rho_s})\cap V=W$$
	
\end{proposition}

\begin{proof}
The tangent space $T_s\Scusp$ is a subspace of $\left<\frac{\partial}{\partial x_1},\frac{\partial}{\partial y_1},\frac{\partial}{\partial x_2},\frac{\partial}{\partial y_2}\right>$. From the previous paragraph, we know that $\left[\frac{\partial}{\partial x_1}\right],\left[\frac{\partial}{\partial y_1}\right],\left[\frac{\partial}{\partial x_2}\right],\left[\frac{\partial}{\partial y_2}\right]\in H^1_{\rho_s}(\Z^2,\so(3,1)^{\rhohyp})$. It follows that
$$W\subset H^1(\Z^2,\so(3,1)^{\rho_s}).$$

Next, let $v\in V$ and write 
$$v=\dot a\left[D_a\right]+\dot b\left[D_b\right]+w,$$
where $w\in W$. Since $\g=\so(3,1)\oplus \v$ (see \eqref{cohomsplitting}) it follows from Lemma \ref{vbasis} that $v\in H^1(\Z^2,\so(3,1)^{\rho_s})$ if an only if $\dot a=\dot b=0$, or in other words if $v\in W$.

\end{proof}

The following proposition shows that at the level of tangent spaces the only redundancy in $\Sc$ up to conjugacy comes from representations having the same cusp shape.

\begin{proposition}\label{sliceintersection}
	Let $s\in C$, $z=\CS(s)$, $C_z=\CS^{-1}(z)$, and $\Scusp_z$ be the image of $C_z$ under $F$ then 
	$$B^1(\Z^2,\g^{\rho_s})\cap T_s\Sc=T_s\Scusp_z.$$  
\end{proposition}
\begin{proof}

	Let $w\in B^1(\Z^2,\g^{\rho_s})\cap T_s\Sc$ and write 
	$$w=\dot{a}\frac{\partial}{\partial a}+\dot{b}\frac{\partial}{\partial b}+\dot{x_1}\frac{\partial}{\partial x_1}+\dot{y_1}\frac{\partial}{\partial y_1}+\dot{x_2}\frac{\partial}{\partial x_2}+\dot{y_2}\frac{\partial}{\partial y_2}$$
	
	Suppose for the sake of contradiction that $\dot a\neq 0$. Looking at the (2,2) entries of $\tilde\rho_s(\gamma_1)$ and $\tilde\rho_s(\gamma_2)$ we see that $e^{3ax_1/4}$ and $e^{3ax_2/4}$ are eigenvalues of the respective matrices. Since $w$ is tangent to a conjugacy path we see that $e^{3ax_1/4}$ and $e^{3ax_2/4}$ must remain constant up to first order. Since $s\in C$ this implies that $\dot ax_1=\dot a x_2=0$. Since $\dot a\neq 0$ this implies that $x_1=x_2=0$. However this contradicts the fact that $s\in C\subset S$, and so $\dot a=0$. A similar argument shows that $\dot b=0$.
	
	Since $\dot a=\dot b=0$ it follows that $w\in T_s\Scusp$. As previously mentioned, $\CS$ is a conjugacy invariant and it follows that $\CS$ is constant to first order in the direction of $w$, and so $w\in T_s\Scusp_z$.  
	
	On the other hand, suppose that $w\in T_s\Scusp_z$. Clearly, $w\in T_s\Sc$, and so we must show that $w$ is tangent to a path of conjugations.
	 By conjugating $\tilde\rho_s$ by a rotation, we can assume without loss of generality that $s=(0,0,x_1,1/x_2,x_2,0)$. From the proof of Lemma \ref{csderiv} we see that $w\in \ker\nabla f(s)\cap \ker \nabla g(s)\cap \ker \nabla h(s)$, and computing the relevant derivatives gives 
	$w=c(0,0,-1/x_2,x_1,0,x_2),$
	for some $c\in \R$. Next, let 
	$$R_\theta=\begin{pmatrix}
		1 & 0 & 0 & 0\\
		0 & \cos \theta & -\sin \theta & 0\\
		0 & \sin \theta & \cos \theta & 0\\
		0 & 0 & 0 & 1
	\end{pmatrix}.$$
	Conjugating by $R_{c\theta}$ and taking the derivative with respect to $\theta$ at 0 gives
	$$\left .\frac{d}{d\theta}\right|_{\theta=0}R_{c\theta}\rho_s(\gamma_1)R_{c\theta}^{-1}=\begin{pmatrix}
		0 & -c/x_2 & cx_1 & 0\\
		0 & 0 & 0 & -c/x_2\\
		0 & 0 & 0 & cx_1\\
		0 & 0 & 0 & 0
	\end{pmatrix},\ \left .\frac{d}{d\theta}\right|_{\theta=0}R_{c\theta}\rho_s(\gamma_2)R_{c\theta}^{-1}=\begin{pmatrix}
		0 & 0 & cx_2 & 0\\
		0 & 0 & 0 & 0\\
		0 & 0 & 0 & cx_2\\
		0 & 0 & 0 & 0
	\end{pmatrix}.$$
	As a result we see that the tangent vector to this conjugacy path is
	 $$(\dot a,\dot b,\dot x_1,\dot y_1,\dot x_2,\dot y_2)=(0,0,-c/x_2,cx_1,0,cx_2)=w.$$
	 Thus we see that $w$ is the tangent vector to a path of conjugations and so $w\in B^1(\Z^2,\g^{\rho_s})\cap T_s\Sc$ 
\end{proof}

Proposition \ref{sliceintersection} has the following immediate corollary.
\begin{corollary}\label{sliceinhomol}
	If $s\in C$ the image of $T_s\Sc$ in $H^1(\Z^2,\g^{\rho_s})$ is 4-dimensional. 
\end{corollary}

\section{The transversality argument}\label{transargument}

Recall that $M$ is a finite volume hyperbolic 3-manifold with $k\geq 1$ cusps, $\Gamma$ is its fundamental group, $\{\Delta_1,\ldots,\Delta_k\}$ is a collection of peripheral subgroups, one for each cusp, $G=\SL_4(\R)$ and $\rhohyp:\Gamma\to\SO(3,1)\subset G$ is the holonomy of its complete hyperbolic structure. 

The main goal of this section is to prove Theorem \ref{mainthm2}. The proof of Theorem \ref{mainthm2} has two parts. First, we use a transversality argument involving the slice from Section \ref{slicesection} to produce a $k$-dimensional family of deformations of $\rhohyp$ in $\hom(\Gamma,G)$ whose image in $\rep(\Gamma,G)$ is also $k$ dimensional. Specifically, we prove:

\begin{theorem}\label{repdef}
	Let $M$ be a finite volume, non-compact hyperbolic 3-manifold with $k\geq 1$ cusps. Suppose that $M$ is infinitesimally rigid rel. $\partial M$ then  there is a $k$-dimensional subspace $V$ of $H^1(\Gamma,\g^{\rhohyp})$, a neighborhood, $U$ of $0$ in $V$, and a smooth family of representations $\mathcal{F}=\{\rho_u\mid u\in U\}$ in $\hom(\Gamma,G)$ such that 
		\begin{itemize}
		\item $\rho_0=\rhohyp$
		\item  For each $u\in U$, $\rho_u\vert_{\Delta_i}$ is the holonomy of a type 0, type 1, or type 2 generalized cusp.
	\end{itemize} 
	Furthermore, if $[\mathcal{F}]$ is the image of $\mathcal{F}$ in $\rep(\Gamma,G)$ then the Zariski tangent space to $[\mathcal{F}]$ at $[\rhohyp]$ is $V\subset H^1(\Gamma,\g^{\rhohyp})$.
\end{theorem}

Next, we apply Theorem \ref{cltdefthm} which guarantees that the representations produced in Theorem \ref{repdef} are holonomies of properly convex projective structures on $M$. We can now prove Theorem \ref{mainthm2} modulo Theorem \ref{repdef}.

\begin{proof}[Proof of Theorem \ref{mainthm2} modulo Theorem \ref{repdef}.]
	By Theorem \ref{repdef}, the restriction $\rho_u\in \mathcal{F}$ to each peripheral subgroup is the holonomy of a generalized cusp of type 0, type 1, or type 2. In particular the peripheral subgroups are virtual flag groups. By Theorem \ref{cltdefthm} we see that after possibly shrinking $U$ we can assume $\mathcal{F}\subset \hom_{ce}(\Gamma,G)$. Furthermore, since the Zariski tangent space to $[\mathcal{F}]$ in $\rep(\Gamma,G)$ at $[\rhohyp]$ is a $k$-dimensional subspace in $H^1(\Gamma,\g^{\rhohyp})$, we see that $[\mathcal{F}]$ is  $k$-dimensional. Again after possibly shrinking $U$, we conclude that $[\mathcal{F}]$ is the image of a $k$-dimensional family of convex projective structures in $\mathfrak{B}(M)$. 
	
	\end{proof}

\subsection{Proof of Theorem \ref{repdef}}

The remainder of this section is dedicated to the proof of Theorem \ref{repdef}. We now briefly describe a strategy to construct such a family of representations in Theorem \ref{repdef}. For the sake of simplicity, briefly assume that $M$ has a single cusp and that $\rhohyp$ has been conjugated so that $\res(\rhohyp)\in \Sc$. First, we show that near $\rhohyp$,  $\res$ is an immersion from $\hom(\Gamma,G)$ to $\hom(\Delta,G)$ whose image has codimension 3. Next, we show that $\res$ is transverse to $\Sc$. As mentioned before $\Sc$ has dimension 5 and hence codimension $13$ in $\hom(\Z^2,G)$. Thus the intersection of $\Sc$ and the image of $\res$ is a 2-dimensional submanifold. However, by Proposition \ref{sliceintersection}, only 1 of these dimensions is accounted for by conjugacy, and so there must be a path $\rho_t:\Gamma\to G$ of pairwise non-conjugate representations.

We now describe the details of the above construction. The overall strategy is similar to that found in the construction of convex projective structures found in \cite{BDL}. For this reason we will quote various results from this work.

When addressing the case of multiple cusps (i.e.\ $k>1$) one quickly encounters the the following problem: While the restriction map $\res:\hom(\Gamma,G)\to \hom(\Delta,G)$ is an immersion, its codimension is too large (it is $18k-15$ rather than $3k$). Recall that the group $\Delta=\oplus_{i=1}^k \Delta_i$, where the $\Delta_i$ are the fundamental groups of the boundary components fo $M$. Roughly speaking, this extra codimension is coming from the fact that we are not able to conjugate the restrictions of a representation to each peripheral subgroup independently. To cope with this problem, we construct an \emph{augmented restriction map} that allows us to perform these independent conjugacies. Let $M$ be a finite volume hyperbolic 3-manifold with fundamental group $\Gamma$ and $k$ cusps. Define $\widetilde\hom(\Gamma,G):=\hom(\Gamma,G)\times G^{k-1}$ and let
$$\widetilde \res:\widetilde\hom(\Gamma,G)\to \hom(\Delta,G),$$
by $(\rho,g_2,\ldots, g_k)\mapsto (\res_1(\rho),g_2\cdot \res_2(\rho),\ldots ,g_k\cdot \res_{k}(\rho))$, where the action of $G$ on $\hom(\Delta,G)$ is the adjoint action. Observe, that when $k=1$ that $\widetilde \res=\res$. The main result concerning the augmented restriction map is that locally it is a submersion with the desired codimension.

\begin{theorem}[Thm 3.8 in \cite{BDL}]\label{augrest}
	Let $M$ be a finite volume hyperbolic manifold with $k\geq 1$ cusps and fundamental group $\Gamma$, and let $\rhohyp$ be the holonomy of the complete hyperbolic structure. Suppose further that $M$ is infinitesimally rigid rel.\ $\partial M$. Then for any $(g_2,\ldots, g_k) \in G^{k-1}$, $\widetilde \res$ is a local submersion onto a submanifold of codimension $3k$ near $(\rhohyp, g_2,\ldots, g_k)$
\end{theorem}

Picking generators $\gamma_1^i$ and $\gamma_2^i$ for $\Delta_i$, we let $\Sc_i$ be the copy of $\Sc$ in $\hom(\Delta_i,G)$, let $\Scusp_i$ be the copy of $\Scusp$ in $\Sc_i$, let $\Sigma=\Sc_1\times \ldots \times \Sc_k$, and let $\Sigma_c=\Scusp_1\times \ldots \times \Scusp_k$. Choose $g_i\in G$ so that $g_i\cdot \res_i(\rhohyp)\in \Sc_i$. Furthermore, by choosing $s_i\in \Scusp_i$ we can arrange that $\rho_{s_i}=\res_i(\rhohyp)$. For $s=(s_1,\ldots, s_k)$, let $V_\Sigma$ be the image of $T_s\Sigma$ in $H^1(\Delta,\g^{\rhohyp})$. 

In this context we can prove the following transversality result involving $\widetilde \res$ and $\Sigma$. 

\begin{proposition}\label{transversalityprop}
	The map $\widetilde \res$ is transverse to $\Sigma$ at $(\rhohyp,g_2,\ldots, g_k)$, with $2k$-dimensional local intersection. 
\end{proposition}


In order to prove Proposition \ref{transversalityprop} we need the following Lemma:

\begin{lemma}\label{translemma}
	 $$H^1(\Delta,\g^{\rhohyp})=V_\Sigma\oplus \res_\ast(H^1(\Gamma,\so(3,1)^{\rhohyp})).$$

Moreover, if $L=\res_\ast(H^1(\Gamma,\g^{\rhohyp}))\cap V_\Sigma$ then $\pi_{\v}\vert_{L}$ is an isomorphism between $L$ and $\res_\ast(H^1(\Gamma,\v^{\rhohyp}))$.    
\end{lemma}
\begin{proof}

 Lemmas \ref{so31cohomology} and \ref{vcohomology} imply that $H^1(\Delta,\g^{\rhohyp})$ is $6k$-dimensional and that $\res_\ast(H^1(\Gamma,\so(3,1)^{\rhohyp}))$ is $2k$-dimensional. Corollary \ref{sliceinhomol} implies that $V_\Sigma$ is $4k$-dimensional, and so the result will follow if we can show that $V_\Sigma\cap \res_\ast(H^1(\Gamma,\so(3,1)^{\rhohyp}))$ is trivial.

 Let $[w]\in V_\Sigma\cap \res_\ast(H^1(\Gamma,\so(3,1)^{\rhohyp}))$, then there is $[w']\in H^1(\Gamma,\so(3,1)^{\rhohyp})$ so that $\res_\ast([w'])=[w]$. For each $1\leq i\leq k$, choose a non-trivial element $m_i\in \Delta_i$, and let $\mu_i$ be the subgroup generated by $m_i$. As in Section \ref{3mfldCohomology}, let $\mu$ be the direct sum of the $\mu_i$ and let $H^1(\mu,\so(3,1)^{\rhohyp})$ be the corresponding cohomology group. The restriction of $[w]$ to $\Delta_i$ is contained in $V_\Sigma\cap H^1(\Delta_i,\so(3,1)^{\rhohyp})$ and is hence in the span of the set $\left\{\left[\frac{\partial}{\partial x_1}\right],\left[\frac{\partial}{\partial y_1}\right],\left[\frac{\partial}{\partial x_2}\right],\left[\frac{\partial}{\partial y_2}\right]\right\}$. It follows that there is $p_i=(u_i,v_i)\in \R^2\bs\{0\}$ such that $w(m_i)=c_{p_i},$ where 
 $$c_{p_i}=\begin{pmatrix}
 	0 & u_i & v_i & 0\\
 	0 & 0 & 0 & u_i\\
 	0 & 0 & 0 & v_i\\
 	0 & 0 & 0 & 0
 \end{pmatrix}$$
 
 It is then easy to check that $w$ is a coboundary when restricted to $\mu_i$, and hence $\res_\ast([w'])=0\in H^1(\mu,\so(3,1)^{\rhohyp})$. However, by Lemma \ref{so31cohomology}, $\res_\ast:H^1(\Gamma,\so(3,1)^{\rhohyp})\to H^1(\mu,\so(3,1)^{\rhohyp})$ is injective and so $[w']=0$. Since $[w]=\res_\ast([w'])$ it follows that $[w]=0$.

 For the last point, the image of $\pi_\v$ restricted to $L$ is contained in $\res_\ast(H^1(\Gamma,\v^{\rhohyp}))$. The kernel of $\pi_\v\vert_{L}$ is easily seen to be $\res_\ast(H^1(\Gamma,\so(3,1)^{\rhohyp}))\cap V_\Sigma$, and so by the previous argument, $\pi_{\v}\vert_{L}$ is an injection.  
 By the previous transversality, $L$ is $k$-dimensional and by Lemma \ref{vbasiseq}, $\res_\ast(H^1(\Gamma,\v^{\rhohyp}))$ is also $k$-dimensional, and so for dimensional reasons $\pi_\v\vert_{L}$ is an isomorphism.  \end{proof}

We can now prove Proposition \ref{transversalityprop}.

\begin{proof}[Proof of Proposition \ref{transversalityprop}]
	
	Let $\vec{\rho}_{hyp}=(\res_1(\rhohyp),\ldots,\res_k(\rhohyp))$. Near $\vec{\rho}_{hyp}$, the space $\hom(\Delta,\g^{\rhohyp})$ is $18k$-dimensional. By construction, $\Sigma$ has codimension $13k$ and contains $\vec{\rho}_{hyp}$. Let $I$ be the image of $\widetilde \res$, then by Theorem \ref{augrest}, $I$ has codimension $3k$ near $\vec{\rho}_{hyp}$. Thus if the intersection of $\Sigma$ and $I$ is transverse at $\vec{\rho}_{hyp}$ then the intersection will have codimension $16k$, or equivalently dimension $2k$. 
	
	The tangent space to $\hom(\Delta,G)$ at $\vec{\rho}_{hyp}$ is $Z^1(\Delta,\g^{\rhohyp})$ and we can write $$Z^1(\Delta,\g^{\rhohyp})\cong H^1(\Delta,\g^{\rhohyp})\oplus B^1(\Delta,\g^{\rhohyp}).$$
	 From the construction of $\widetilde\res$, it can be seen that at $p=(\rho,g_2,\ldots,g_k)\in \widetilde\hom(\Gamma,G)$,
	 $$T_p\left(\widetilde\hom(\Gamma,G)\right)\cong Z^1_{\rho}(\Gamma,\g^{\rhohyp})\oplus\left(\oplus_{i=2}^kB^1_{g_i\cdot \rho}(\Gamma,\g^{\rhohyp})\right),$$
	 and that the map $\widetilde\res_\ast:T_p\left(\widetilde\hom(\Gamma,G)\right)\to Z^1_{\rho}(\Delta,\g^{\rhohyp})$ is just the componentwise application of $\res_\ast$.  Since $\rhohyp$ is an irreducible representation, it follows that $B^1(\Delta,\g^{\rhohyp})\subset T_{\vec{\rho}_{hyp}}I$. On the other hand, from Lemma \ref{translemma} we know that $V_\Sigma$ and $\res_\ast(H^1(\Gamma,\so(3,1)^{\rhohyp})$ span $H^1(\Delta,\g^{\rhohyp})$. As a result, $T_{\vec{\rho}_{hyp}}I$ and $T_{\vec{\rho}_{hyp}}\Sigma$ span $Z^1(\Delta,\g^{\rhohyp})$, and are thus transverse. 
	\end{proof}

We can now prove Theorem \ref{repdef}. 

\begin{proof}[Proof of Theorem \ref{repdef}.]

Recall that there are $g_i\in G$ so that $g_i\cdot \res_i(\rhohyp)\in \Sc_i$, let $p=(\rhohyp,g_2,\ldots,g_k)\in \widetilde\hom(\Gamma,G)$, and let $p'=\widetilde\res(p)$. By Lemma \ref{translemma}, $\res_\ast(H^1(\Gamma,\g^{\rhohyp}))$ intersects $V_\Sigma$ transversely in a $k$-dimensional subspace $\tilde V$. Let $V$ be the $k$-dimensional subspace of $H^1(\Gamma,\g^{\rhohyp})$ such that $\res_\ast(V)=\tilde V$.

  As a result, we can find a lift $R:V\to Z^1(\Delta,\v^{\rhohyp})$ of $\res_\ast$ such that 
$$R(H^1(\Gamma,\v^{\rhohyp}))\subset W:=\widetilde\res_\ast(T_p(\widetilde\hom(\Gamma,G)))\cap T_{p'}(\Sigma).$$

In other words, there is a commutative the diagram 
$$
\begin{tikzcd}
 & W\subset Z^1(\Delta,\g^{\rhohyp})\arrow[d]\\
 V\arrow[ur,"R"]\arrow[r,"\res_\ast"] & H^1(\Delta,\g^{\rhohyp}) 	
\end{tikzcd}
$$
 The space $W$ is the tangent space to the intersection of the image of $\widetilde\res$ and $\Sigma$ at $p'$. Thus by Proposition \ref{transversalityprop} we can find a small neighborhood, $U$, of $0$ in $V$ and

\begin{itemize}
	\item A smooth family $\mathcal{F}=\{\rho_u\mid u\in U\}$ of representation in $\hom(\Gamma,G)$ such that $\rho_0=\rhohyp$.  The tangent space of $\res(\mathcal{F})$ at $\res(\rhohyp)$ is $R(V)$.
	\item  Smooth families $\{g^u_i\mid u\in U\}$ of elements of $G$ for $2\leq i\leq k$, such that $g_i^0=g_i$ and such that $g_i^u\cdot \res_i(\rho_u)\in \Sc_i$.
\end{itemize}

By construction, the image of the space of infinitesimal deformations of $\mathcal{F}$ at $\rhohyp$ in $H^1(\Gamma,\g^{\rhohyp})$ is $V$, and so $[\mathcal{F}]$ is $k$-dimensional. Furthermore, $\res_i(\rho_u)$ is conjugate into $\Sc_i$. By Theorem \ref{sliceisgencuspholonomy} this implies that the restriction of $\rho_u$ to each peripheral subgroup is the holonomy of a generalized cusp of type 0, type 1, or type 2.

\end{proof}

\section{Controlling the cusps}\label{controlcusps}

In this section we describe some theoretical results that make it possible to control the types of the cusps that are produced by Theorem \ref{mainthm2}. This will allow us to prove Theorem \ref{mainthm}. The first main results of this section is Theorem \ref{slicecoordtyperelationship} which describes a sufficient condition for ensuring that Theorem \ref{mainthm2} produces properly convex manifolds with type 2 cusps. The condition in Theorem \ref{slicecoordtyperelationship} involves the value of certain cohomological quantities. In Section \ref{examples} we calculate these quantities for some examples in order to find explicit manifolds that admit properly convex structures with type 2 cusps. 

The other main result of this section, Theorem \ref{achiralthm}, shows that in the presence of orientation reversing symmetries it is sometimes possible to guarantee that the deformations produced by Theorem \ref{mainthm2} have (some) type 1 cusps. 

\subsection{Slice coordinates for $H^1(\Delta,\v^{\rhohyp})$}

 Recall that $M$ is a finite volume hyperbolic 3-manifold with fundamental group $\Gamma$. The manifold $M$ has $k\geq 1$ cusps $\{\partial_1,\ldots,\partial_k\}$ and assume that we have chosen a peripheral subgroups $\{\Delta_1,\ldots,\Delta_k\}$ one for each cusp. For each $\Delta_i\cong \Z^2$ pick a set $\{\gamma_1^i,\gamma_2^i\}$ of generators. Recall that $H^1(\Delta,\v^{\rhohyp}):=\bigoplus_{i=1}^kH^1(\Delta_i,\v^{\rhohyp})$. The spaces $H^1(\Delta_i,\v^{\rhohyp})$ and $H^1(\Z^2,\v^{\rhohyp})$ are isomorphic vector spaces and we would like to identify a convenient isomorphism between these two spaces. 
 
  By Lemma \ref{vbasis}, $\left\{\left[D_a\right],\left[D_b\right]\right\}$ is a basis of $H^1(\Z^2,\v^{\rhohyp})$. Using the generating set $\{\gamma_1^i,\gamma_2^i\}$, we can identify $\Delta_i$ with $\Z^2$ and $H^1(\Delta_i,\v^{\rhohyp})$ with $H^1(\Z^2,\v^{\rhohyp})$. Next, assume that $\rhohyp$ has been conjugated so that $\rhohyp\vert_{\Delta_i}\in \Sc\subset \hom(\Z^2,G)$.  Using \eqref{alphamat} and \eqref{betamat}, define cocycles $z_{\gamma_1^i}$ and $z_{\gamma_2^i}$ in $Z^1(\Delta_i,\v^{\rhohyp})$ by the property that $z_{\gamma_1^i}(\gamma_1^i)=\alpha_1$, $z_{\gamma_1^i}(\gamma_2^i)=\alpha_2$, $z_{\gamma_2^i}(\gamma_1^i)=\beta_1$, $z_{\gamma_2^i}(\gamma_2^i)=\beta_2$. It is easy to see that $\{[z_{\gamma_1^i}],[z_{\gamma_2^i}]\}$ is a basis for $H^1(\Delta_i,\v^{\rhohyp})$. A basis constructed in this way is called a \emph{slice basis} for $H^1(\Delta_i,\v^{\rhohyp})$ (with respect to $\{\gamma_1^i,\gamma_2^i\}$). If we regard elements of a slice basis as elements of $H^1(\Delta,\v^{\rhohyp})$ then $\left\{[z_{\gamma_1^1}],[z_{\gamma_2^i}],\ldots,[z_{\gamma_1^k}],[z_{\gamma_2^k}]\right\}$ is a basis for $H^1(\Delta,\v^{\rhohyp})$, which we also call a \emph{slice basis}.

  Suppose now that $M$ is infinitesimally rigid rel.\ $\partial M$.  Recall that $V=\res_\ast^{-1}(\res_\ast(H^1(\Gamma,\g^{\rhohyp})\cap V_\Sigma))$ and observe that since $M$ is infiniteismally rigid rel.\ $\partial M$ that $V$ is a $k$-dimensional subspace. If $[z]\in V$ then Lemma \ref{translemma} implies that  $\pi_{\v}\circ\res_\ast([z])$ is a non-trivial element of $H^1(\Delta,\v^{\rhohyp})$ and so we can write
  
  \begin{equation}\label{zslicecoords}
  \pi_{\v}\circ\res_\ast([z])=c_{\gamma_1^1}[z_{\gamma_1^1}]+c_{\gamma_2^1}[z_{\gamma_2^1}]+\ldots+ c_{\gamma_1^k}[z_{\gamma_1^k}]+c_{\gamma_2^k}[z_{\gamma_2^k}]	
  \end{equation}
\noindent as a non-trivial linear combination. 
  
  The coordinates of $\pi_{\v}\circ\res_\ast([z])$ with respect to the slice basis coming from \eqref{zslicecoords} are called \emph{slice coordinates} for $[z]$.  The next Theorem describes the relationship between the slice coordinates and the cusp types of the properly convex manifolds produced by Theorem \ref{mainthm}.
  
 	\begin{theorem}\label{slicecoordtyperelationship}
 		Suppose that $M$ is infinitesimally rigid rel.\ $\partial M$, and suppose $[z]\in V$ has slice coordinates $\left(c_{\gamma_1^1},c_{\gamma_2^1},\ldots,c_{\gamma_1^k},c_{\gamma_2^k}\right)$. Let $I_{[z]}=\{i\mid c_{\gamma_1^i}\neq 0 {\rm\ and\ }c_{\gamma_2^i}\neq 0\}$ and $II_{[z]}=\{i\mid c_{\gamma_1^i}\neq 0 {\rm\ or\ }c_{\gamma_2^i}\neq 0\}$ 
 		\begin{enumerate}
 		\item $M$ admits a convex projective structure where if $i\in I_{[z]}$ then the $i$th cusp is a type 2 generalized cusp and	
 		\item $M$ admits a convex projective structure where if $i \in II_{[z]}$ then the $i$th cusp is a type 1 or a type 2 generalized cusp for each $i\in II_{[z]}$
 		\end{enumerate}
 		\end{theorem}
 		
 		\begin{proof}
 		To minimize notation we address the case when $M$ has a single cusp. The multiple cusp case can be treated similarly. From Theorem \ref{repdef}, for each $[z]\in V$ there is a family $\rho_t:\pi_1M\to G$ of representations such that $\rho_0=\rhohyp$ and whose Zariski tangent vector is $z$. Furthermore $\rho_t\vert_\Delta$ is a path in $\Sc$ with Zariski tangent vector $w=\res_\ast(z)$.  As such we can write 
 		$$w=\dot{a}D_a+\dot{b}D_b+\tilde w,$$
 		where $\tilde w\in Z^1(\Delta,\so(3,1)^{\rhohyp})$, and observe that this implies that $\dot a$ and $\dot b$ are the slice coordinates of $[z]$. 
 		 		
 		If either $\dot{a}$ or $\dot{b}$ is non-zero then by examining \eqref{slicerep} it follows that as $t$ moves away from 0 some eigenvalue of either $\rho_t(\gamma_1^1)$ or $\rho_t(\gamma_2^1)$ is changing to first order in $t$ away from 1. This implies that for $t\neq 0$ that $\rho_t$ is the holonomy of either a type 1 or type 2 cusp, which proves the second claim. Similarly, if both $\dot{a}$ and $\dot{b}$ are non-zero, it follows that as $t$ moves away from 0 that two eigenvalues of both $\rho_t(\gamma_1^1)$ and $\rho_t(\gamma_2^1)$ are changing to first order in $t$ away from 1. This implies that for $t\neq 0$ that $\rho_t$ is the holonomy of a type 2 cusp, which proves the first claim. 		
 		 \end{proof}
 		 
 \begin{remark}
 It is easy to see that if $\rho_t$ is a path in $\Sc$ that is type 0 when $t=0$ and type 1 otherwise that the Zariski tangent vector to this path at $t=0$ has either $\dot{a}$ or $\dot{b}=0$, but not both. It is tempting to say that if $i\in II\backslash I$ then the $i$th cusp is type 1. However, this turns out not to be the case. The problem is that the slice coordinates are only encoding first order behavior. For instance, the representations 
 $$\rho_t(\gamma_1^1)=\exp\begin{pmatrix}
0 & 1 & 0 & 0\\
0 & t & 0 & 1\\
0 & 0 & 0 & 0\\
0 & 0 & 0 & 0	
\end{pmatrix},\ 
\rho_t(\gamma_2^1)=\exp\begin{pmatrix}
	0 & 0 & 1 & 0\\
	0 & 0 & 0 & 0\\
	0 & 0 & t^2 & 1\\
	0 & 0 & 0 & 0
\end{pmatrix}
$$
	are holonomies of type 2 cusps when $t\neq 0$, however, up to first order the second generator remains constant. 
 \end{remark}

Before proceeding with the proof of Theorem \ref{mainthm} we need to recall the following result from \cite{BDL} that will ensure that we can find a cohomology class in $H^1(\Gamma,\v^{\rhohyp})$ whose restriction to each cusp is non-trivial. We will use this result to ensure that the representations we construct in Theorem \ref{repdef} will be holonomies of type 1 or 2 cusps rather than type 0 (standard hyperbolic cusps).

\begin{lemma}[See Lem 4.4 of \cite{BDL}]\label{eachcuspnontriv}
There exists a cohomology class $[z]\in H^1(\Gamma,\v^{\rhohyp})$ with the property that $(\res_i)_\ast[z]\in H^1(\Delta,\v^{\rhohyp})$ is non-trivial for each $1\leq i\leq k$. 	
\end{lemma}

We can now prove Theorem \ref{mainthm}.

\begin{proof}[Proof of Theorem \ref{mainthm}]
By Lemma \ref{eachcuspnontriv} we can find a cohomology class $[w]\in H^1(\Gamma,\v^{\rhohyp})$ with the property that $(\res_i)_\ast[w]\in H^1(\Delta_i,\v^{\rhohyp})$ is non-trivial for each $1\leq i\leq k$. Furthermore, by Lemma \ref{translemma} there is $[z]\in V$ such that $\pi_{\v}\circ \res_\ast([z])=[w]$. It follows that for each $1\leq i\leq k$ that either $c_{\gamma_1^i}$ or $c_{\gamma_2^i}$ is non-zero and thus $II_{[z]}=\{1,\ldots,k\}$. Applying Theorem \ref{slicecoordtyperelationship}(2) gives the desired conclusion.
\end{proof}

\subsection{Symmetry and type 1 cusps}

One consequence of Theorem \ref{slicecoordtyperelationship} is that if the slice coordinates of a cohomology class $[z]\in H^1(\Gamma,\v^{\rhohyp})$ are all non-zero then the resulting convex projective structures corresponding to $[z]$ have all type 2 cusps. A priori, a vector having all non-zero coordinates seems like a generic condition and so it is natural to wonder if Theorem \ref{mainthm2} ever produces examples with type 1 cusps. In this section we show that in certain circumstances it is possible to produce examples with type 1 cusps. More specifically, we prove a general result (Theorem \ref{achiralthm}) which says that for manifolds admitting certain types of symmetry, Theorem \ref{mainthm} produces convex projective manifolds where some of the cusps become type 1 generalized cusps. In Section \ref{examples} we use Theorem \ref{achiralthm} to show that if $K$ is the $6_3$ knot then $M=S^3\backslash K$ admits a properly convex projective structure where the cusp is type 1.

Before proceeding with the proof we discuss how orientation reversing symmetries of $M$ act on $\partial M$ and on $H^1(\Delta,\v^{\rhohyp})$. Let $\phi:M\to M$ be an orientation reversing symmetry and define 
$${S_\phi=\{i\in \{1,\ldots,k\}\mid \phi(\partial_i)=\partial_i\}}$$
 to be the set of cusps invariant under $\phi$. 

We first need to address some technicalities regarding how $\phi$ induces an action on the peripheral subgroups. The map $\phi$ gives rise to an outer automorphism $[\phi_\ast]\in \operatorname{Out}(\Gamma):=\operatorname{Aut}(\Gamma)/\operatorname{Inn}(\Gamma)$. We now describe how $[\phi_\ast]$ induces an action on $\Delta_i$ for each $i\in S_\phi$. Let $\phi_1,\phi_2\in [\phi_\ast]$, and so there is $g\in \Gamma$ such that $\phi_2(\gamma)=g\phi_1(\gamma)g^{-1}$ for any $\gamma\in \Gamma$. Since $\phi(\partial_i)=\partial_i$ there are $g_1,g_2\in \Gamma$ such that $g_j\phi_j(\Delta_i)g_j^{-1}=\Delta_i$ for $j\in \{1,2\}$. For $j\in \{1,2\}$, composing $\phi_j$ with conjugation by $g_j$ gives an automorphism of $\Delta_i$ and we claim that this map is independent of the choice of $\phi_j$ and $g_j$. To see this, observe that $g_1(g_2g)^{-1}$ normalizes $\Delta_i$. Since $\Gamma$ is the fundamental group of a finite volume hyperbolic 3-manifold the normalizer of $\Delta_i$ in $\Gamma$ is equal to the centralizer of $\Delta_i$ in $\Gamma$. This implies that $g_1(g_2g)^{-1}$ centralizes $\Delta_i$, and thus conjugation by $g_1$ and by $g_2g$ give rise to the same map from $\phi_1(\Delta_i)$ to $\Delta_i$. As a result, 
$$g_2\phi_2(\gamma)g_2^{-1}=g_2g\phi_1(\gamma)g_{-1}g_2^{-1}=g_1\phi_1(\gamma)g_1^{-1},$$
which proves the claim. By abuse of notation we will call this map $\phi_\ast:\Delta_i\to\Delta_i$. 

\begin{lemma}\label{involutionlemma}
	Let $\phi:M\to M$ be an orientation reversing symmetry and let $i\in S_\phi$ then $\phi:\partial_i\to\partial_i$ is isotopic to an involution. Furthermore, there exists a generating set $\{\gamma_+^i,\gamma_-^i\}$ for $\pi_1(\partial_i)$ such that $\phi_\ast(\gamma_{\pm}^i)=(\gamma_{\pm}^i)^{\pm 1}$.  
\end{lemma}

\begin{proof}
	Since $\phi_\ast$ is an automorphism of $\Delta_i$ and $\Delta_i\cong \Z^2$, $\phi_\ast$ corresponds to an element $M_\phi\in \GL(2,\Z)$. Since $\phi$ is orientation reversing, it follows that $\det(M_\phi)=-1$. As such the characteristic polynomial of $M_\phi$ is $p_\phi(x)=x^2-\tr(M_\phi)-1$.

	By Mostow rigidity, the mapping class group of $M$ is finite, and so $M_\phi$ is a finite order element of $\GL(2,\Z)$. It follows that the roots, $\lambda_1,\lambda_2$, of $p_\phi(x)$ are roots of unity. Suppose that the roots of $p_\phi$ are non-real, then $\lambda_2=\overline{\lambda_1}$. However, since $\det(M_\phi)=-1$, we see that $-1=\lambda_1\overline{\lambda_1}=\abs{\lambda_1}^2$, which is a contradiction. Thus the roots of $p_\phi$ are real. Since $\lambda_1$ and $\lambda_2$ are real and $\det(M_\phi)=-1$, we find that $\{\lambda_1,\lambda_2\}=\{-1,1\}$. It follows $p_\phi(x)=x^2-1$, and is hence $\phi_\ast$ is an involution. Since $\GL(2,\Z)$ is the mapping class group of $\partial_i$ it follows that $\phi$ is isotopic to an involution when restricted to $\partial_i$. 
	
	It is clear that there are non-trivial $\pm 1$ eigenspace for the action of $\phi_\ast$ on $H^1(\partial_i,\R)$, and the proof will be complete if it can be shown that there are eigenvectors in $H^1(\partial_i,\Z)\cong \Delta_i$. Let $\tilde\gamma_+^i$ be a non-trivial element of $\Delta_i$ that is not a $-1$-eigenvector of $\phi_\ast$. By the Cayley-Hamilton theorem, $M_\phi$ is a root of its characteristic polynomial, and so $\gamma_+^i=(M_\phi+\operatorname{I})\tilde\gamma_+^i$ is a non-trivial $1$-eigenvector of $M_\phi$. Using a similar procedure we can construct a non-trivial $-1$-eigenvector $\gamma_{-}^i$. The set $\{\gamma_+^i,\gamma_-^i\}$ is the desired generating set. Using a similar construction we can produce the appropriate generating sets for the remaining $\phi$-invariant cusps, thus completing the proof.
	 \end{proof}

	The generators $\{\gamma_+^i,\gamma_-^i\}$ constructed in Lemma \ref{involutionlemma}  are called the \emph{$p$-curve and $m$-curve of the $i$th cusp with respect to $\phi$} (or simply the \emph{$p$-curve and $m$-curve} is $\phi$ and $i$ are clear from context). 
	
	If $i\in S_\phi$ then $\phi_\ast$ is an involution when restricted to $\Delta_i$. It is natural to wonder if $\phi$ induces an involution on $H^1(\Delta_i,\v^{\rhohyp})$. Strictly speaking, $\phi$ does not induce an action, but instead induces a map
	
	$$\phi^\ast:H^1(\Delta_i,\v^{\rhohyp})\to H^1(\Delta_i,\v^{\rhohyp\circ \phi_\ast})$$
	
	 However, by Mostow rigidity, there is a unique $A_\phi\in O(3,1)$ such that $\rhohyp\circ \phi_\ast=A_\phi\cdot \rhohyp$. Conjugation by $A_\phi^{-1}$ provides a map 
$$\Ad A_{\phi}^{-1}:H^1(\Delta_i,\v^{\rhohyp\circ \phi_\ast})\to H^1(\Delta_1,\v^{\rhohyp}).$$
Composing these two maps gives an automorphism of $H^1(\Delta_1,\v^{\rhohyp})$ which by abuse of notation we refer to as $\phi^\ast$. In the same way, we can view $\phi^\ast$ as an automorphism of $H^1(\Gamma,\v^{\rhohyp})$.

	It turns out that $\phi^\ast$ does act as an involution and that using $\gamma_+^i$ and $\gamma_-^i$ we can construct a nice eigenbasis for $H^1(\Delta_i,\v^{\rhohyp})$. We will need the following Lemma, which shows that if $M$ admits an orientation reversing symmetry then the cusp shape of an invariant cusp with respect to $\{\gamma_-^i,\gamma_+^i\}$ is purely imaginary. This is originally due to an observation of Riley \cite{Riley}.
	
	\begin{lemma}\label{pureimagcuspshape}
	Let $\phi:M\to M$ be an orientation reversing symmetry and let $i\in S_\phi$. Then the cusp shape of $\partial_i$ with respect to $\{\gamma_-^i,\gamma_+^i\}$ is $z=ic$, where $c>0$. Consequently, it is possible to conjugate $\rhohyp$ in $G$ so that 
\begin{equation}\label{imaginaryCuspShape}\rhohyp(\gamma_-^i)=\begin{pmatrix}
	1 & 1 & 0 & 1/2\\
	0 & 1 & 0 & 1\\
	0 & 0 & 1 & 0\\
	0 & 0 & 0 & 1
\end{pmatrix},{\rm\ and\ }
\rhohyp(\gamma_+^i)=\begin{pmatrix}
	1 & 0 & c & c^2/2\\
	0 & 1 & 0 & 0\\
	0 & 0 & 1 & c\\
	0 & 0 & 0 & 1
\end{pmatrix}
\end{equation}
	\end{lemma}
	
	\begin{proof}
		We can regard $\rhohyp$ as a representation from $\Gamma$ to $\PSL(2,\C)$ in such a way that 
		$$
		\rhohyp(\gamma_-^i)=\begin{pmatrix}
			1 & 1 \\
			0 & 1
		\end{pmatrix}{\rm\ and\ }
		\rhohyp(\gamma_+^i)=\begin{pmatrix}
		1 & z\\
		0 & 1	
\end{pmatrix},
$$
where $z$ is the cusp shape with respect to $\{\gamma_-^i,\gamma_+^i\}$, which by construction has positive imaginary part. By Mostow rigidity, there is an element $B_\phi\in PSL(2,\C)$ such that for each $\gamma\in \Gamma$, $\rhohyp(\phi_\ast(\gamma))=\overline{B_\phi\rhohyp(\gamma)B^{-1}_\phi}$, where $\overline{g}$ means entrywise complex conjugation. Since $\phi_\ast(\gamma_\pm^i)=(\gamma_{\pm}^i)^{\pm}$, it follows that $-z=\overline{z}$. In other words, $z$ is purely imaginary and thus $z=ic$ for $c>0$.
	\end{proof}

 Next, define two cocycles $z_+^i$, $z_-^i\in Z^1(\Delta_1,\v^{\rhohyp})$, by $z_{\pm}^i(\gamma_{\mp})=0$ and $z_{\pm}^i(\gamma_{\pm}^i)=a_{\pm}$, where
$$a_+=\begin{pmatrix}
	-c & & & \\
	   & -c & \\
	 & & 3c & \\
	 & & & -c
\end{pmatrix},{\rm\ and\ }
a_-=\begin{pmatrix}
	-1 & & & \\
	   & 3 & \\
	 & & -1 & \\
	 & & & -1
\end{pmatrix}.
$$
The following proposition shows that $[z_\pm^i]$ are nothing more than the slice basis for $H^1(\Delta_i,\v^{\rhohyp})$ with respect to $\{\gamma_+^i,\gamma_-^i\}$. 

\begin{proposition}
	The cohomology classes $[z_{\pm}^i]$ are the slice basis for $H^1(\Delta_i,\v^{\rhohyp})$ with respect to $\{\gamma_-^i,\gamma_+^i\}$
\end{proposition}

\begin{proof}
	To simplify notation we drop the $i$ superscripts. Since $\rhohyp(\gamma_-)$ and $\rhohyp(\gamma_+)$ are of the form \eqref{imaginaryCuspShape} it follows that $D_a(\gamma_+)-z_-(\gamma_+)=0$ and 
	$$D_a(\gamma_-)-z_-(\gamma_-)=\begin{pmatrix}
		0 & \frac{1}{4} & 0 & 0\\
		0 & 0 & 0 & -\frac{1}{4} \\
		0 & 0 & 0 & 0\\
		0 & 0 & 0 & 0
	\end{pmatrix}$$
	
	It follows that $(D_a-z_-)(\gamma)=v-\rhohyp(\gamma)\cdot v$, where 
	$$v=\begin{pmatrix}
		\frac{1}{16} & \frac{1}{8} & 0 & 0\\
		0 & -\frac{3}{16} & 0 & -\frac{1}{8}\\
		0 & 0 & \frac{1}{16} & 0\\
		0 & 0 & 0 & \frac{1}{16}
	\end{pmatrix}$$
\noindent It follows that $[D_a]=[z_-]$. A similar argument shows that $[D_b]=[z_+].$
\end{proof}

We can now show that $[z_\pm^i]$ are the desired eigenvectors of $\phi^\ast$.

\begin{lemma}\label{evectors}
 Suppose that $\phi_\ast:\Delta_i\to\Delta_i$ is induced by an orientation reversing symmetry of $M$ as above. Then $[z_\pm^i]$ is a $\pm 1$-eigenvector for the action of $\phi^\ast$ on $H^1(\Delta_1,\v^{\rhohyp})$. Furthermore, $\{[z_+^i],[z_-^i]\}$ is an eigenbasis for $H^1(\Delta_i,\v^{\rhohyp})$. 	
\end{lemma}

\begin{proof}
Again, to simplify notation we drop the the $i$ scripts on the $z_\pm$, $\gamma_{\pm}$, and $\Delta$. First, since $\phi_\ast$ leaves $\Delta$ invariant and $\phi_\ast(\gamma_{\pm})=\gamma_{\pm}^{\pm 1}$ it follows that there are $x,y\in \R$ so that

$$A_{\phi}=\begin{pmatrix}
	1 & -x & y &\frac{x^2+y^2}{2}\\
	0 & -1 & 0 & x\\
	0 & 0 & 1 & y\\
	0 & 0 & 0 & 1
\end{pmatrix}
$$

We first show that $[\phi^\ast(z_+)]=[z_+]$. By the discussion above we see that at the level of cocyles that $(\phi^\ast(z_+))(\gamma_-)=0$. Furthermore,
$$\phi^\ast(z_+)(\gamma_+)=A_\phi^{-1}\cdot z_+(\phi(\gamma_+))=A^{-1}_\phi\cdot z_+(\gamma_+)=A_\varphi^{-1}\cdot a_+.$$
Therefore, 
$$(\phi^\ast(z_+)-z_+)(\gamma_+)=A_{\phi}^{-1}\cdot a_+-a_+=\begin{pmatrix}
	0 & 0 & -4y & -4y^2\\
	0 & 0 & 0 & 0\\
	0 & 0 & 0 & 4y\\
	0 & 0 & 0 & 0
\end{pmatrix}.$$
We now show that $(\phi^\ast(z_+)-z_+)$ is a coboundary. Let 
$$v_+=\begin{pmatrix}
	-\frac{y}{c} & 0 & -\frac{2y(c+y)}{c} & 0\\
	0 & -\frac{y}{c} & 0 & 0\\
	0 & 0 & \frac{3y}{c} & \frac{2y(c+y)}{c}\\
	0 & 0 & 0 & -\frac{y}{c}	
\end{pmatrix}
$$
Consider the coboundary $w_+(\gamma)=v_+-\rhohyp(\gamma)\cdot v_+$. Computing, one sees that $\rhohyp(\gamma_-)$ commutes with $v_+$ and so $w_+(\gamma_-)=0$, and also that 
$$w_+(\gamma_+)=v_+-\rhohyp(\gamma_+)\cdot v_+=\begin{pmatrix}
	0 & 0 & -4y & -4y^2\\
	0 & 0 & 0 & 0\\
	0 & 0 & 0 & 4y\\
	0 & 0 & 0 & 0
\end{pmatrix},$$
and so $w_+=\phi^\ast(z_+)-z_+$. Thus $[z_+]$ is a 1-eigenvector of $\phi^\ast$. 

The other case is similar. Computing shows that $(\phi^\ast(z_-)+z_-)(\gamma_+)=0$. Using the cocycle condition gives $z_-(\gamma_-^{-1})=-\rhohyp(\gamma_-)^{-1}\cdot z_-(\gamma_-)$ we find that 
$$(\phi^\ast(z_-)+z_-)(\gamma_-)=A_{\phi}^{-1}\cdot z_-(\gamma_-^{-1})+z_-(\gamma_-)=-A_{\phi}^{-1}\rhohyp(\gamma_-)^{-1}\cdot a_-+a_-=\begin{pmatrix}
			0 & -4(1+x) & 0 & 4(1+x)^2\\
			0 & 0 & 0 & 4(1+x)\\
			0 & 0 & 0 & 0\\
			0 & 0 & 0 & 0
		\end{pmatrix}.
$$ 
Let 
$$v_-=\begin{pmatrix}
	-1-x & 2x(1+x) & 0 & 0\\
	0 & 3+3x & 0 & -2x(1+x)\\
	0 & 0 & -1-x & 0\\
	0 & 0 & 0 & -1-x 
\end{pmatrix}$$
and let $w_-(\gamma)=v_--\rhohyp(\gamma)\cdot v_-$, then as before we see that $w_-=\phi^\ast(z_-)+z_-$, and so $[z_-]$ is a $-1$-eigenvector of $\phi^\ast$. Finally, $H^1(\Delta_1,\v^{\rhohyp})$ is a 2-dimensional vector space and $[z_\pm]$ are non-trivial eigenvectors with different eigenvalue and so they must be linearly independent, and hence a basis. 	
\end{proof}

	 We can now state the main theorem of this section that describes when a manifold admitting an orientation reversing symmetry admits a convex projective structure with type 1 cusps. 
	 
	 \begin{theorem}\label{achiralthm}
Let $M$ be an infinitesimally rigid rel.\ $\partial M$ and let $\phi:M\to M$ be an orientation reversing symmetry that leaves each cusp invariant. If $\phi^\ast:H^1(\Gamma,\v^{\rhohyp})\to H^1(\Gamma,\v^{\rhohyp})$ is the identity map then there is a properly convex projective structure on $M$ where each cusp is type 1.
\end{theorem}

Theorem \ref{achiralthm} has the following Corollary. 

\begin{corollary}\label{achiralcor}
	Suppose that $M$ has a single cusp and that $\phi:M\to M$ is an orientation reversing symmetry and that $\gamma_+$ is a $p$-curve for $\phi$. If $M$ is infinitesimally rigid rel.\ $\partial M$ and the map $\res_\ast:H^1(\Gamma,\v^{\rhohyp})\to H^1(\gamma_+,\v^{\rhohyp})$ is nontrivial then $M$ admits nearby convex projective structures where the cusp is a type 1 generalized cusp. 
\end{corollary}

\begin{proof}
Since $M$ has a single cusp, $\Delta=\Delta_1$, it follows trivially that each cusp is preserved by $\phi$. Since $M$ is infinitesimally rigid rel.\ $\partial M$ it follows that $H^1(\Gamma,v)$ is $1$-dimensional. The image of $H^1(\Gamma,\v^{\rhohyp})$ in $H^1(\Delta,\v^{\rhohyp})$ is $\phi^\ast$ invariant, and is thus spanned by either $[z_+]$ or $[z_-]$. By hypothesis, $\res_\ast:H^1(\Gamma,\v^{\rhohyp})\to H^1(\gamma_+,\v^{\rhohyp})$ is non-trivial, but the image of $[z_-]$ is trivial in $H^1(\langle\gamma_+\rangle,\v^{\rhohyp})$, and thus $\res_\ast(H^1(\Gamma,\v^{\rhohyp})$ is spanned by $[z_+]$. It follows that $\phi^\ast:H^1(\Gamma,\v^{\rhohyp})\to H^1(\Gamma,\v^{\rhohyp})$ is the identity. The result follows by applying Theorem \ref{achiralthm}. 

\end{proof}

 Before proving Theorem \ref{achiralthm} we need a couple of auxiliary lemmas. The first Lemma allows us to identify $\res_\ast(H^1(\Gamma,\v^{\rhohyp}))$ inside $H^1(\Delta,\v^{\rhohyp})$. 

\begin{lemma}\label{phiaction}

Suppose that $M$ is infinitesimally rigid rel.\ $\partial M$ and let $\phi:M\to M$ be an orientation reversing symmetry that preserves each cusp. Then $\phi^\ast:H^1(\Gamma,\v^{\rhohyp})\to H^1(\Gamma,\v^{\rhohyp})$ is an involution. Moreover, there is an eigenbasis $\B=\{v_1,\ldots,v_k\}$ such that $\res_\ast(v_i)=[z_{e(i)}^i]$, where $e:\{1,\ldots,k\}\to \{+,-\}$ is the function that returns $+1$ if $v_i$ is a $+1$-eigenvector and $-$ if $v_i$ is a $-1$-eigenvectors. 
\end{lemma}

\begin{proof}
	Recall that $H^1(\Delta,\v^{\rhohyp})=\bigoplus_{i=1}^kH^1(\Delta_i,\v^{\rhohyp})$. We have the commutative diagram
	\begin{equation}\label{evalcd}
		\begin{tikzcd}
		H^1(\Gamma,\v^{\rhohyp})\arrow[d,"\phi^\ast"']\arrow[r,"\res_\ast"] & H^1(\Delta,\v^{\rhohyp})\arrow[d,"\phi^\ast"]\\
		H^1(\Gamma,\v^{\rhohyp})\arrow[r,"\res_\ast"] & H^1(\Delta,\v^{\rhohyp})
	\end{tikzcd}
	\end{equation}

	 It follows that the image of $H^1(\Gamma,\v^{\rhohyp})$ in $H^1(\Delta_1,\v^{\rhohyp})$ is $\phi^\ast$-invariant. By hypothesis, $H^1(\Gamma,\v^{\rhohyp})$ is $k$-dimensional and $\res_\ast$ is injective and so there is a basis $\B'=\{v_1',\ldots,v_k'\}$ for $\res_\ast(H^1(\Gamma,\v^{\rhohyp}))$ consisting of vectors from the set $\{[z_+^1],[z_-^1],\ldots,[z_+^k],[z_-^k]\}$ of $\phi^\ast$-eigenvectors of $H^1(\Delta,\v^{\rhohyp})$. Let $v_i=\res_\ast^{-1}(v_i')$. From \eqref{evalcd}, it follows that $\B=\{v_1,\ldots,v_k\}$ is an eigenbasis consisting of $\pm1$ eigenvectors and $\phi^\ast:H^1(\Gamma,\v^{\rhohyp})\to H^1(\Gamma,\v^{\rhohyp})$ is thus an involution.
	 
	  Next, suppose that for some $1\leq i\leq k$ that $[z_+^i],[z_-^i]\in \B'$, then by the pigeonhole principal there must be $j\neq i$ such that 
	 $$(\res_j)_\ast:H^1(\Gamma,\v^{\rhohyp})\to H^1(\Delta_j,\v^{\rhohyp})$$
	  is trivial, however, this contradicts Lemma \ref{eachcuspnontriv}. Thus by renumbering the elements of $\B$ we can ensure that $\res_\ast(v_i)=[z_{e(i)}^i]$

\end{proof}

The next Lemma gives a sufficient condition for a representation in $\Sc$ to be the holonomy of a type 0 or type 1 cusp. This criteria will be used in the proof of Theorem \ref{achiralthm}

\begin{lemma}\label{type1criteria}
	Let $A\in \GL(2,\Z)$ be such that $A(\gamma_1)=\gamma_1^{-1}$ and $A(\gamma_2)=\gamma_2$ and let $s_0=(0,0,1/c,0,0,c)$. There is a neighborhood $V$ of $s_0$ in $S$ with the property that if $v\in V$ and $\rho_v\in \Sca$ is such that $\rho_v\circ A$ is conjugate to $\rho_v$ then $\rho_v$ is the holonomy of a type 0 or type 1 generalized cusp. 	
\end{lemma}

\begin{proof}
Let $s=(a,b,x_1,y_1,x_2,y_2)\in S$ and suppose that $\rho_s\circ A$ is conjugate to $\rho_s$. Let $\{ax_1,by_1\}$ and $\{ax_2,by_2\}$ be the set of two eigenvalues of largest modulus of $m^1_s$ and $m_2^s$, respectively.  Since $\rho_s$ is conjugate to $\rho_s\circ A$ we find that either 
 \begin{gather}
 	ax_1=-ax_1 {\rm\ and\ } by_1=-by_1 {\rm\ or}\label{eq1}\\
 	-ax_1=by_1 {\rm\ and\ } ax_2=by_2\label{eq2}
 \end{gather}
Since $s\in S$, equations \eqref{eq1} imply that either $a$ or $b$ is zero. We can choose $V$ such that $x_1\neq 0$ for all $s\in V$. In this case equations \eqref{eq2} imply that $b(x_1y_2+y_1x_2)=0$. By further shrinking $V$ we can assume that $x_1y_2+y_1x_2\neq 0$, and thus \eqref{eq2} implies that $b=0$. Thus in either case we see that for $s\in V$ that $\rho_s$ is the holonomy of either a type 0 or type 1 generalized cusp. 
   
\end{proof}

We can now prove Theorem \ref{achiralthm}.

\begin{proof}[Proof of Theorem \ref{achiralthm}]
	Before proceeding with the details, we describe the idea behind the proof. Let $\mathcal{F}$ be the $k$-dimensional family of representations produced in Theorem \ref{repdef}, and recall that $[\mathcal{F}]$ is the image of $\mathcal{F}$ in $\rep(\Gamma,G)$. We begin by showing that near $[\rhohyp]$, elements $[\rho_u]\in [\mathcal{F}]$  are determined by the eigenvalues of $\rho_u(\gamma_+^i)$. By construction, $\rho_u\circ \phi_\ast(\gamma_+^i)$ is conjugate to $\rho_u(\gamma_+^i)$ for $1\leq i\leq k$ and therefore $\rho_u\circ \phi_\ast$ is conjugate to $\rho_u$. Restricting $\rho_u$ to each cusp and applying Lemma \ref{type1criteria} gives the desired result. We now provide the details of this argument. 
	
	The symmetry $\phi$ acts on $\widetilde \hom(\Gamma,G)$ by $\phi\cdot(\rho,g_2,\ldots,g_k)=(\rho\circ \phi_\ast,g_2,\ldots,g_k)$ and on $\hom(\Delta,G)$ by $\phi\cdot(\rho_1,\ldots, \rho_k)= (\rho_1\circ \phi_\ast,\ldots, \rho_k\circ \phi_\ast)$. A simple computation shows that $\widetilde \res $ is equivariant with respect to these actions.  Another simple computation shows that $\Sigma$ is $\phi$-invariant. Combining these facts we find that $\mathcal{F}\subset \hom(\Gamma,G)$ is also $\phi$-invariant. Moreover, the action of $\phi$ descends to $\rep(\Gamma,G)$ and the above computation shows that $[\mathcal{F}]$ is also $\phi$-invariant. Furthermore, by Mostow rigidity $\phi\cdot [\rhohyp]=[\rhohyp]$. 
	
	 By hypothesis, $\phi^\ast:H^1(\Gamma,\v^{\rhohyp})\to H^1(\Gamma,\v^{\rhohyp})$ is the identity. Combining this with Lemma \ref{phiaction} shows that $\res_\ast(H^1(\Gamma,\v^{\rhohyp}))$ is spanned by $\{[z_+^1],\ldots,[z_+^k]\}$. Let $U\ni 0$ be the neighborhood in $V$ used to define $\mathcal{F}$ in Theorem \ref{repdef} and let $(t_1,\ldots t_k)$ be coordinates on $U$ such that if $u=(t_1,\ldots,t_k)$ then up to conjugacy,
	 
	 \begin{gather}\rho_u(\gamma_-^i)=\exp\begin{pmatrix}\label{mat-i}
	 	O(\abs{u}^2) & \ast & \ast & \ast\\
	 	0 & O(\abs{u}^2) & 0 & \ast\\
	 	0 & 0 & O(\abs{u}^2) & \ast\\
	 	0 & 0 & 0 & O(\abs{u}^2)
	 \end{pmatrix},\\
	 \rho_u(\gamma_+^i)=\exp\begin{pmatrix}\label{mat+i}
	 	-t_i+O(\abs{u}^2) & \ast & \ast & \ast\\
	 	0 & -t_i+O(\abs{u}^2) & 0 & \ast\\
	 	0 & 0 & 3t_i+O(\abs{u}^2) & \ast\\
	 	0 & 0 & 0 & -t_i+O(\abs{u}^2)
	 \end{pmatrix}
	 \end{gather} 
	 
	 In other words the partial derivative of $\rho_u$ with respect to  $t_i$ at $\rhohyp$, when projected to  $H^1(\Delta_i,\v^{\rhohyp})$, is $[z_+^i]$. The matrix $\rho_u(\gamma_+^i)$ has a unique simple eigenvalue when $t_i\neq 0$ which we denote $v_i$. From \eqref{mat+i} it follows that $v_i=\exp(3t_i+O(\abs{u}^2))=1+3t_i+O(\abs{u}^2)$. Thus by the inverse function theorem the map $(v_1,\ldots,v_k)\mapsto (t_1,\ldots,t_k)$ gives a diffeomorphism from a neighborhood of $(1,\ldots,1)\in \R^k$  to a neighborhood $\tilde U$ of $[\rhohyp]$ in $[\mathcal{F}]$. 
	 
	   Let $[\rho_u]\in \tilde U \subset [\mathcal{F}]$, and let $\rho_u\in \mathcal{F}$ be a representative of this conjugacy class. By construction, $\phi_\ast(\gamma_+^i)=\gamma_+^i$, for $1\leq i\leq k$. It follows that $\rho_u\circ\phi_\ast(\gamma_+^i)$ is conjugate to $\rho_u(\gamma_+^i)$. Since $\tilde U$ is parameterized by $v_i$ for $1\leq i\leq k$ it follows that $\rho_u\circ \phi_\ast$ is conjugate to $\rho_u$. By applying Lemmas \ref{pureimagcuspshape} and \ref{type1criteria} we see that (after possibly shrinking $\tilde U$) $\res_i(\rho_u)$ is the holonomy of a type 0 or type 1 cusp for each $[\rho_u]\in \tilde U$.

	 Finally, by Lemma \ref{eachcuspnontriv} there is $[w]\in H^1(\Gamma,\v^{\rhohyp})$ whose restriction to each cusp is non-trivial and by Lemma \ref{translemma} there is $[z]\in V$ such that $\pi_{\v}\circ \res_\ast([z])=[w]$ . If we let $\rho_t$ be a path through $\rhohyp$ in $\mathcal{F}$ tangent to $z$ then Theorem \ref{slicecoordtyperelationship} implies that these representations will be the holonomies of properly convex structures on $M$ with type 1 cusps.

\end{proof}

\subsection{Calculating slice coordinates}
 
 In order to apply Theorem \ref{slicecoordtyperelationship} it is necessary to be able to calculate the slice coordinates, or at least decide when they are non-zero. We close this section with a discussion about calculating slice coordinates using more easily accessible data. Recall that the Lie algebra $\g$ admits a \emph{Killing form} $B:\g\otimes \g\to \R$ be the given by $a\otimes b\mapsto 4\tr(ab)$.  The Killing form is easily seen to be invariant under the adjoint action of $G$ on $\g$. If $\Pi$ is the fundamental group of a closed $n$-manifold and $\rho:\Pi\to G$ is a representation, the Killing form gives rise to the \emph{Poincar\'e duality pairing}
 \begin{equation}\label{poincareduality}
 	H^p(\Pi,\g^{\rho})\otimes H^{n-p}(\Pi,\g^{\rho})\stackrel{\cup}{\to} H^n(\Pi,\g^{\rho}\otimes \g^{\rho})\stackrel{B}{\to}H^n(\Pi,\R)\cong \R
 \end{equation}
 
 It is easy to check that the pairing in \eqref{poincareduality} respects the splitting $\g\cong \so(3,1)\oplus \v$ and so we get
 
 \begin{equation}\label{poincaredualityv}
 	H^p(\Pi,\v^{\rho})\otimes H^{n-p}(\Pi,\v^{\rho})\stackrel{\cup}{\to} H^n(\Pi,\v^\rho\otimes \v^{\rho})\stackrel{B}{\to}H^n(\Pi,\R)\cong \R
 	\end{equation}
 	
 	In both cases, the Poincar\'e duality pairing is non-degenerate. We will only have occasion to use this pairing in the simple setting where $n=1$, in which case the construction can be made quite explicit. Specifically, let $\Pi\cong \Z$ will be generated by the homotopy class $\gamma$ of a closed loop in $\partial M$. In this case $H^0(\Pi,\v^{\rho})$ can be identified with the $\rho(\gamma)$-invariant elements of $\v$, which we henceforth denote $\v^{\rho(\gamma)}$. If $[w]\in H^1(\Pi,\v^{\rho})$ and $a\in \v^{\rho(\gamma)}$ then $\langle [w],a\rangle=4\tr (w(\gamma)a)$. We will now use these pairings to calculate slice coordinates. 
 
 Once again, for simplicity, assume that $M$ has a single cusp and that $\{\gamma_1,\gamma_2\}$ is a generating set for $\Delta$. By conjugating we can assume that there are $(u,v)\in \R^2$ with $v> 0$ such that 
 $$\rhohyp(\gamma_1)=\begin{pmatrix}
 	1 & 1 & 0 & 1/2\\
 	0 & 1 & 0 & 1\\
 	0 & 0 & 1 & 0\\
 	0 & 0 & 0 & 1
 \end{pmatrix},\ \rhohyp(\gamma_2)=\begin{pmatrix}
 	1 & u & v & \frac{u^2+v^2}{2}\\
 	0 & 1 & 0 & u\\
 	0 & 0 & 1 & v\\
 	0 & 0 & 0 & 1
 \end{pmatrix}$$
 
 In this setting the complex number $u+iv$ is the cusp shape of the $\partial_1$ with respect to the generating set $\{\gamma_1,\gamma_2\}$. Let $[z]\in H^1(\Gamma,\v^{\rhohyp})$ and assume without loss of generality that $\res_\ast[z]=c_{\gamma_1}[z_{\gamma_1}]+c_{\gamma_2}[z_{\gamma_2}]$. From \eqref{alphamat} and \eqref{betamat} it follows that 
 $$z(\gamma_1)=\begin{pmatrix}
 	-c_{\gamma_1} & \ast &\ast &\ast \\
 	 & 3c_{\gamma_1} & & \ast\\
 	  & & -c_{\gamma_1} & \ast\\
 	   &  & & -c_{\gamma_1}
 \end{pmatrix},\ z(\gamma_2)=\begin{pmatrix}
 	-c_{\gamma_1}u-c_{\gamma_2}v &\ast  &\ast &\ast \\
 	& 3c_{\gamma_1}u-c_{\gamma_2}v & &\ast \\
 	& & 3c_{\gamma_2}v-c_{\gamma_1}u &\ast \\
 	& & & -c_{\gamma_1}u-c_{\gamma_2}v
 \end{pmatrix}$$
 
 Next, let
 
 $$ 
 \delta_{u,v}=\begin{pmatrix}
 	-1 & & & \\
 	 & -\frac{u^2-3v^2}{u^2+v^2} & -\frac{4uv}{u^2+v^2} & \\
 	 & -\frac{4uv}{u^2+v^2} & \frac{3u^2-v^2}{u^2+v^2} & \\
 	 & & & -1
 \end{pmatrix}
 $$
 
 It is easily checked that $\delta_{1,0}\in \v^{\rhohyp(\gamma_1)}$ and $\delta_{u,v}\in \v^{\rhohyp(\gamma_2)}$. Restricting to the subgroup generated by $\gamma_1$ (resp.\ $\gamma_2$) allows one to regard $[z]$ as an element of $H^1(\langle \gamma_1\rangle,\v^{\rhohyp})$ (resp.\ $H^1(\langle \gamma_2\rangle,\v^{\rhohyp})$), and computing pairings we find that 
 \begin{equation}\label{poincarepairings}
 	{d_1:=\langle [z],\delta_{1,0}\rangle=-16c_{\gamma_1}},\  {d_2:=\langle [z],\delta_{u,v}\rangle=-16\frac{c_{\gamma_1}u(u^2-3v^2)+c_{\gamma_2} v(v^2-3u^2)}{u^2+v^2}}
 \end{equation}
 
 In other words, there is a linear relationship between the slice coordinates and the parings $d_1$ and $d_2$ and this linear relationship is encoded by the matrix

 $$M(u,v)=-16\begin{pmatrix}
 	1 & 0\\
 	\frac{u(u^2-3v^2)}{u^2+v^2} & \frac{v(v^2-3u^2)}{u^2+v^2}
 \end{pmatrix},
$$

  By extending this discussion to the multiple cusp setting we can define the pairings $d_1^j$ and $d_2^j$ and the matrix $M(u_j,v_j)$, for $1\leq j\leq k$ and to arrive at the following Proposition
 
 \begin{proposition}\label{coefcalc}
 	Suppose that $M$ has $k$ cusps and is infinitesimally rigid rel.\ $\partial M$. Let $w_j=u_j+iv_j$ be the cusp shape of the $j$th cusp of $M$ with respect to $\{\gamma_1^j,\gamma_2^j\}$, and let $\mathcal{M}$ be the block matrix given by 

\begin{equation}\label{coefmat}
	\mathcal{M}=\begin{pmatrix}
	M(u_1,v_1) & & \\
	 & \ddots & \\
	  & & M(u_k,v_k)
\end{pmatrix}
\end{equation}
Let $\vec c=(c_{\gamma_1^1},c_{\gamma_2^1},\ldots, c_{\gamma_1^k},c_{\gamma_2^k})$ and let $\vec{d}=(d_1^1,d_2^1,\ldots,d_1^k,d_2^k)$. If $\arg(w_j)\notin \frac{\pi}{3}\Z$ for each $1\leq j\leq k$ then $M$ is invertible and 
$$\mathcal{M}^{-1}\vec d=\vec c$$

 \end{proposition}

 \begin{proof}
The matrix $\mathcal{M}$ is invertible iff $M(u_j,v_j)$ is invertible for each $1\leq j\leq k$. By examining determinants, it follows that $M(u_j,v_j)$ is singular if and only if $v_j^2-3u_j^2=0$. The equation $v_j^2-3u_j^2=0$ is satisfied iff $v_j=\pm\sqrt{3}u_j$ iff $\frac{v_j}{u_j}=\pm\sqrt{3}$. Since $\tan(\arg(u_j+iv_j))=\frac{v_j}{u_j}$ it follows that $M$ is singular if and only if $\arg(u_j+iv_j)\in \frac{\pi}{3}\Z$, thus by hypothesis, $\mathcal{M}$ is invertible. 

By the discussion of the previous paragraph, $\mathcal{M}\vec c=\vec d$, and the result follows.  
 \end{proof}

 \begin{remark}\label{cuspshape}
 By changing generating set for $\Delta_i$, it is always possible to ensure that no cusps shape has argument that is an integral multiple of $\pi/3$. 	
 \end{remark}

\section{Examples}\label{examples}

This section is dedicated to producing explicit examples of 1-cusped manifolds where Theorem \ref{mainthm} produces both type 1 and type 2 cusps. Specifically in Section \ref{5_2section}, we show that if $M=S^3\backslash K_{5_2}$ where $K_{5_2}$ is the $5_2$ knot (see Figure \ref{5_2}), then $M$ admits a family of convex projective structures where the cusp is a type 2 generalized cusp. Then, in Section \ref{6_3section} we show that if $M=S^3\bs K_{6_3}$, where $K_{6_3}$ is the $6_3$ knot (see Figure \ref{6_3}) then $M$ admits a convex projective structure where the cusp is a type 1 generalized cusp.

 \subsection{The $5_2$ knot complement}\label{5_2section}

\begin{center}
	\begin{figure}
	\includegraphics[scale=.25]{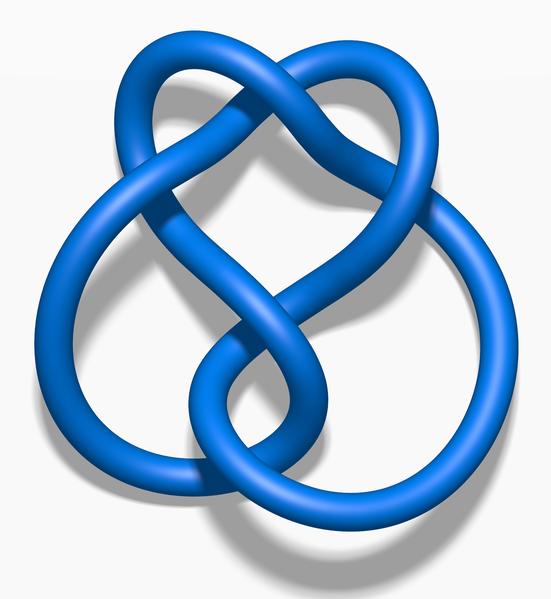}	
	\caption{\label{5_2}The $5_2$ knot}
	\end{figure}

\end{center}

Before proceeding we mention that there are other recent examples of manifolds admitting type 2 cusps due to Martin Bobb \cite{Bobb}, however his examples involve a version of bending for arithmetic manifolds. By work of Reid \cite{Reid}, the figure-eight knot is the only arithmetic knot complement, and so we see that the examples covered in this section are non-arithmetic and hence not covered by Bobb's work. 

Let $M_1=S^3\bs K_1$ where $K_1$ is the $5_2$ knot, let $\Gamma_1=\pi_1M_1$, let $\rhohyp:\Gamma_1\to G$ be the holonomy of the complete hyperbolic structure on $M_1$, and let $\Delta\cong \Z^2$ be a peripheral subgroup of $\pi_1(M_1)$. In order to apply Theorem \ref{mainthm} we first need to check that $M_1$ is infinitesimally rigid rel.\ $\partial M_1$. 

\begin{proposition}\label{5_2rigid}
	$H^1(\Gamma_1,\v^{\rhohyp})$ is 1-dimensional. In particular, $M_1$ is infinitesimally rigid rel.\ $\partial M_1$. 
\end{proposition}

\begin{proof}
	The proof is computational and consists of computing the rank of a certain matrix with entries in a number field.  This computation has been implemented in Sage \cite{sage} and can be found along with a detailed explanation in the Sage notebook {\tt 5\_2rigid.ipynb}  that can be found at \cite{comps}. We now outline some of the relevant details.
	
	Let $\Gamma_1=\pi_1M_1$, then 
	$$\Gamma_1=\langle x,y\mid xwy^{-1}w^{-1}=1\rangle,$$
	where $w=yxy^{-1}x^{-1}yx$
	
	Let $r=xwy^{-1}w^{-1}$, then there is an $\R$-linear map $\v\times\v\to \v$ given by $(a,b)\mapsto \frac{\partial r}{\partial x}\cdot a+\frac{\partial r}{\partial y}\cdot b$, where $\frac{\partial r}{\partial x}$ and $\frac{\partial r}{\partial y}$ are Fox derivatives and the action of $\Z[\Gamma_1]$ on $\v$ is given by composing $\rhohyp$ with the adjoint action of $\SO(3,1)$ on $\v$. The kernel of this map is naturally isomorphic to the space $Z^1(\Gamma_1,\v^{\rhohyp})$ of 1-cocycles. In {\tt 5\_2rigid.ipynb} the rank of this map is computed to be 8. Since $\dim(\v^{\rhohyp})=9$  this implies that $Z^1(\Gamma_1,\v^{\rhohyp})$ is 10-dimensional.
	
	 The representation $\rhohyp$ is well known to be irreducible, which implies that $B^1(\Gamma_1,\v^{\rho})\cong \v$. This space of coboundaries thus has dimension 9, and hence $H^1(\Gamma_1,\v^{\rhohyp})$ is 1-dimensional. Furthermore, by Lemma \ref{vcohomology}, the image of in $H^1(\Delta,\v^{\rhohyp})$ of $H^1(\Gamma_1,\v^{\rhohyp})$ under the map $\res_\ast$ has dimension 1, and thus $\res_\ast:H^1(\Gamma_1,\v^{\rhohyp})\to H^1(\Delta,\v^{\rhohyp})$ is an injection. 
\end{proof}

Let $\gamma_1$ and $\gamma_2$ be the meridian and homologically determined longitude of $5_2$, then it is easily checked (in SnapPy, for instance) that if $z$ is the cusp shape of $M_1$ with respect to this generating set then $z$ is the unique complex root with positive imaginary part of the polynomial $56-4t+2t^2+t^3$. It is again easily checked that argument of this root is not an integral multiple of $\pi/3$, and it follows that we can use Proposition \ref{coefcalc} to calculate the slice coordinates of  $[z]\in H^1(M_1,\v^{\rhohyp})$. 

\begin{lemma}\label{52cohom}
If $[z]$ is a generator of $H^1(M_1,\v^{\rhohyp})$ and $[z]=c_a[D_a]+c_b[D_b]$ then $c_a,c_b\neq 0$. 	
\end{lemma}

\begin{proof}
	The proof is again a computation that involves calculating the matrix $\mathcal{M}$ from \eqref{coefmat} and the pairings $d_1$ and $d_2$ in \eqref{poincarepairings}. This calculation is also implemented in the sage notebook {\tt 5\_2rigid.ipynb} where it is shown that $c_a$ and $c_b$ are both non-zero. 
\end{proof}

Combining these results we are able to prove the following:

\begin{theorem}\label{52type2defs}
	The manifold $M_1$ admits a properly convex projective structure whose end is a type 2 generalized cusp. 
\end{theorem}

\begin{proof}
	By Lemma \ref{52cohom}, the generator $[z]$ of $H^1(\Gamma_1,\v^{\rhohyp})$ can be written as $[z]=c_a[D_a]+c_b[D_b]$, where $c_a,c_b\neq 0$. The result then follows by applying Theorem \ref{slicecoordtyperelationship}.
	\end{proof}

\subsection{The $6_3$ knot complement}\label{6_3section}
\begin{center}
	\begin{figure}
	\includegraphics[scale=.25]{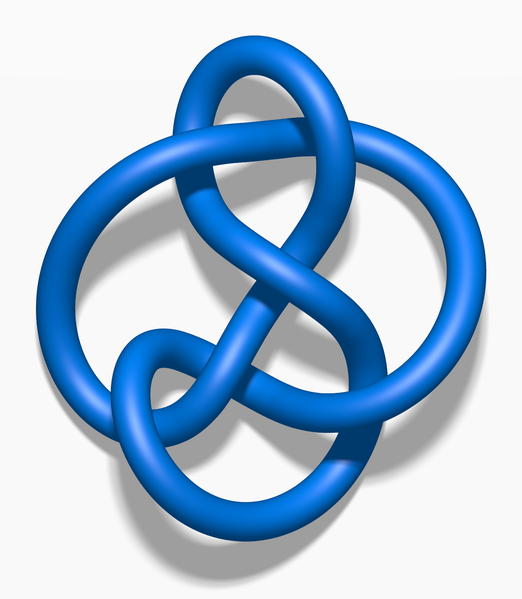}	
	\caption{\label{6_3}The $6_3$ knot}
	\end{figure}
\end{center}

For this section let $M_2=S^3\bs K_2$ where $K_2$ is the $6_3$ knot, let $\Gamma_2=\pi_1M_2$, let $\rhohyp:\Gamma_2\to \SL(4,\R)$ be the holonomy of the complete hyperbolic structure on $M_2$, and let $\Delta\cong \Z^2$ be a peripheral subgroup of $\pi_1(M_2)$.

\begin{proposition}\label{6_3rigid}
	$H^1(\Gamma_2,\v^{\rhohyp})$ is 1-dimensional. In particular, $M_2$ is infinitesimally rigid rel.\ $\partial M_2$. 
\end{proposition}

\begin{proof}
	The proof is essentially the same as that of Proposition \ref{5_2rigid}. The details of the computation can be found in the sage notebook {\tt 6\_3rigid.ipynb} which can be found at \cite{comps}. 
\end{proof}

Using the above result we can prove the following:

\begin{theorem}\label{63type1}
	The manifold $M_2$ admits a properly convex projective structure whose end is a type 1 generalized cusp. 
\end{theorem}

\begin{proof}
	By Proposition \ref{6_3rigid} $M_2$ is infinitesimally rigid rel.\ $\partial M_2$. The knot $6_3$ is a two-bridge knot. It is well known that two-bridge knots are parameterized by a rational number $p/q$, with $p$ odd and that a two-bridge knot is amphicheiral if and only if $p^2=-1\pmod q$. The rational number for the $6_3$ knot is $5/13$, and it is thus amphicheiral.  As a result, $M_2$ admits a symmetry that preserves the homologically determined longitude, $\gamma_+$, and sends the meridian, $\gamma_-$, to its inverse. Furthermore, by the computations in {\tt 6\_3rigid.ipynb} the map $\res_\ast:H^1(\Gamma_2,\v^{\rhohyp})\to H^1(\gamma_+,\v^{\rhohyp})$ is non-trivial. The result then follows by applying Corollary \ref{achiralcor}.
	\end{proof}
	
\begin{remark}\label{fig8comment}
	In \cite{BalFig8}, the author shows, using different methods, that if $K$ is the figure-eight knot then $M=S^3\bs K$ admits a properly convex projective structure with type 1 cusps. However, the figure-eight knot satisfies the hypothesis of Corollary \ref{achiralcor} and so these structures could also be constructed by the methods in this paper. 
\end{remark}
	

\bibliographystyle{plain}
If \bibliography{bibliography}

\end{document}